\documentclass[12pt]{article}
\usepackage{graphicx}
\usepackage[centering, margin={0.8in, 0.5in}, includeheadfoot]{geometry}

\usepackage[utf8x]{inputenc}
\usepackage{amssymb,pdfsync,epstopdf,verbatim,gensymb,amsmath,amsthm}
\usepackage{nccmath, mathtools}
\usepackage{textcomp}
\usepackage{amsmath}
\usepackage{float}
\usepackage{subfig,boldline}
\usepackage{amsfonts}
\usepackage{epsfig,color, hyperref}

\def\R {{\mathbb R}}

\def\H01{{H_0^1(\Omega)}}
\def\L2{{L^2(\Omega)}}

\usepackage{bmpsize}
\newtheorem{theorem}{Theorem}[section] 
\newtheorem{definition}{Definition}[section]
\newtheorem{lemma}{Lemma}[section]

\newtheorem{remark}{Remark}[section]

\newtheorem{proposition}{Proposition}
\newtheorem{algorithm}{Algorithm}[section]

\newcommand{\vp}{\varphi}

\newcommand{\Hc}{\mathcal{H}}

\newcommand{\Rb}{\mathbb{R}}

\newcommand{\Beq}{\begin{equation}}
\newcommand{\Eeq}{\end{equation}}
\newcommand{\beq}{\begin{equation*}}
\newcommand{\eeq}{\end{equation*}}
\newcommand{\bal}{\begin{align}}
\newcommand{\eal}{\end{align}}
\renewcommand{\O}{\Omega}
\renewcommand{\L}{\langle}

\newcommand{\g}{\gamma}

\newcommand{\bp}{\begin{prob}}
\newcommand{\ep}{\end{prob}}
\newcommand{\bpr}{\begin{proof}}
\newcommand{\epr}{\end{proof}}

\newcommand{\bel}[1]{\begin{equation}\label{#1}}
\newcommand{\ee}{\end{equation}}

\graphicspath{{./figs/}}

\title{A sparsity-based nonlinear reconstruction method for
two-photon photoacoustic tomography}
\author{ Madhu Gupta\thanks{Department of Mathematics, University of Texas at Arlington, TX 76019, USA. madhu.gupta@mavs.uta.edu} \and Rohit Kumar Mishra\thanks{Department of Mathematics, University of Texas at Arlington, TX 76019, USA. rohit.mishra@uta.edu} \and Souvik Roy\thanks{Department of Mathematics, University of Texas at Arlington, TX 76019, USA. 	souvik.roy@uta.edu}}

\date{}

\begin{document}
\maketitle
\begin{abstract}
We present a new nonlinear optimization approach for the sparse reconstruction of single-photon absorption and two-photon absorption coefficients in the photoacoustic tomography (PAT). This framework comprises of minimizing an objective functional involving a least squares fit of the interior pressure field data corresponding to two boundary source functions, where the absorption coefficients and the photon density are related through a semi-linear elliptic partial differential equation (PDE) arising in PAT. Further, the objective functional consists of an $L^1$ regularization term that promotes sparsity patterns in absorption coefficients.  The motivation for this framework primarily comes from some recent works related to solving inverse problems in acousto-electric tomography and current density impedance tomography. We provide a new proof of existence and uniqueness of a solution to the semi-linear PDE. Further, a proximal method, involving a Picard solver for the semi-linear PDE and its adjoint, is used to solve the optimization problem. Several numerical experiments are presented to demonstrate the effectiveness of the proposed framework.
\end{abstract}


\section{Introduction}
The hybrid medical imaging problems have attracted the research community a lot in the last few decades. The idea behind hybrid imaging methods is to combine a high contrast modality and a high resolution modality to get images with high contrast and resolution simultaneously. High contrast modalities like electrical impedance tomography (EIT) are used primarily for imaging electrical, optical or elastic properties of biological tissues because these properties vary greatly between healthy and unhealthy tissues. On the other hand, modalities like magnetic resonance imaging (MRI) and ultrasound are used to provide better resolution. Therefore, the inversion process for hybrid imaging problems involves two steps coming from each modality discussed above.  For a more detailed discussion on hybrid imaging techniques, please see the review articles \cite{Bal_Review, Kuchment_Review}.


One of the hybrid imaging modalities is photoacoustic tomography (PAT) that couples electromagnetic waves together with ultrasound. PAT takes advantage of the photoacoustic effect to convert absorbed optical energy into acoustic waves. In PAT, near infrared (NIR) light propagates into a medium of interest and a fraction of the incoming light energy is absorbed, which results in local heating and subsequent cooling of the medium. Due to this heating and cooling phenomenon, acoustic waves are generated that are recorded at the boundary of the medium. The inverse problem is reconstruct the diffusion, absorption and Gr\"{u}neisen coefficients from these acoustic measurements, for more details on the subject see \cite{Ammari_book_2008,Bal_Ren_2011, bal_uhlmann,Kuchment_Kunyansky,Li_Wang,Wang_2004, Wang_2008, Xu_Gaik1, Xu_Gaik2, Xua_Wang} and references therein. 

The PAT technology has two main categories, namely, photoacoustic microscopy (PAM) and photoacoustic computed tomography (PACT). Generally, PAM is known to provide high resolution within a depth of several millimeters. On the other hand, PACT gives a larger penetration depth beyond one centimeter, but at the expense of inferior spatial resolution. To overcome the limitation of PAM, non-linear mechanisms have been introduced such as two-photon absorption \cite{Ren_Zhang,Ren2018,Yamoka}. The phenomenon when an electron transfers to an excited state after simultaneously absorbing two photons can be defined as two-photon absorption. An imaging modality where one tries to recover  optical properties of  heterogeneous media (such as biological tissues) using the photoacoustic effect resulting from two photon absorption is known as two-photon photoacoustic tomography (2P-PAT) \cite{Lai_Lee, Langer, Urban}.  Even though the occurrence of two-photon absorption (in healthy biological tissues)  is less frequent than single-photon absorption, two-photon absorption is extremely useful in practice, see for instance \cite{Denk, Peter, Ying, Webb}.

The mathematical formulation of 2P-PAT was first introduced in \cite{Ren_Zhang,Ren2018}, where the authors consider an optically absorbing and scattering medium $\O \subset \Rb^n\ (n \geq 2)$. Denoting the density of photons at a point $x\in\Omega$ as $u(x)$, it was shown that $u(x)$ solves the following semi-linear diffusion equation
\begin{equation}\label{eq: semi-linear equation}
\begin{aligned}
-\nabla\cdot(D(x)\nabla{u(x)})+\sigma(x) u(x) + \mu(x) |u(x)|u(x)&=0, \quad\qquad &  \mbox{in }\ \ \Omega, \\
u(x)&=g(x), \quad & \mbox{ on } \partial \Omega,
\end{aligned}
\end{equation}
where $D(x)$ denotes the diffusion coefficient, $\sigma(x)$ and $\mu(x)$ represent the single-photon and the two-photon absorption coefficients respectively, and the function $g(x)$ is the illumination pattern on the boundary $\partial \O$. The term $\mu(x)|u(x)|$ is the total two-photon absorption coefficient, where the absolute value of $u$ is taken to ensure that the total two-photon absorption coefficient is non-negative \cite{Ren2018}.

The medium $\O$ heats up due to absorption of some portion of incoming photons that results in thermal expansion of the medium. The medium cools down after photons leave the medium and this results in contraction of the medium, which gives rise to acoustic waves. This effect is known as the photoacoustic effect. This photoacoustic effect generates an acoustic wave pressure field $\Hc^{\sigma, \mu}$ is given by (see \cite{bal_uhlmann, Fisher_Schotland})
\begin{align}
\Hc^{\sigma, \mu}(x) =  \Gamma(x) \left[\sigma(x) u(x) + \mu(x) |u(x)|u(x)\right], \quad \mbox{ for } x \in \Omega,
\end{align}
where $\Gamma$ is the Gr\"{u}neisen coefficient that determines the efficiency of the photoacoustic effect. The aim is to recover the optical properties of the medium $\O$ from the  measured acoustic wave signals on the surface of the medium. In this process, the first step involves the recovery of the initial acoustic wave pressure field $\Hc^{\sigma, \mu}$ from measured data, as usually done in a standard  PAT. In the second step of 2P-PAT, the goal is to reconstruct the optical coefficients $D$, $\sigma$, $\mu$ and $\Gamma$ from the information of internal data $\Hc^{\sigma, \mu}$. This step is usually known as the quantitative step. Recently, the experimental aspect of 2P-PAT have been studied by several authors and it has been shown that the effect of two-photon absorption can be measured accurately, we refer to \cite{Lai_Lee, Langer,  Urban, Yamoka1, Yamoka2, Yamoka3} for detailed discussions. Thus, we assume that the first step in the 2P-PAT process has been accomplished to obtain the initial acoustic wave pressure field $\Hc^{\sigma, \mu}$. For the second step of recovery of the optical coefficients, detailed mathematical and numerical analysis has been done in very few works \cite{Ren_Zhang, Ren2018, Yu}. It has been shown in \cite{Ren2018} that simultaneous reconstruction of all the four coefficients $D$, $\sigma$, $\mu, \Gamma$ is not possible. In \cite{Ren_Zhang,Ren2018}, the authors show that given $D,\Gamma$, one of $\sigma$ and $\mu$ can be reconstructed with internal data corresponding to one boundary illumination pattern and reconstruction of both coefficients require two sets of 
internal datum. The authors also present two reconstruction algorithms for reconstructing $\sigma,\mu$. 

There are three major drawbacks of the existing reconstruction algorithms for 2P-PAT: First, four sets of internal datum are used for reconstructing two coefficients. While this gives better reconstructions, it is not conforming with the theoretical requirement of only two sets of internal datum. Secondly, in the presence of 5\% noise in the data, the reconstructions of $\mu$ exhibit severe artifacts. Thirdly, there is no evidence of the algorithms performing well to reconstruct complex objects with high contrast such as holes and inclusions. In this article, we aim at using a robust computational framework that has the ability to provide high contrast and high resolution reconstructions of objects with holes and inclusions. The framework is based on a non-linear PDE-constrained optimization technique, developed recently  \cite{Roy_AET, Roy_CDII, Roy_AET2} to study the aforementioned hybrid inverse problem for 2P-PAT. We start by formulating a minimization problem where we aim to determine $\sigma$ and $\mu$ given the interior acoustic wave pressure field $\Hc^{\sigma,\mu}$. Additionally, we also assume that the variations in the values of absorption coefficients from known background absorption coefficients demonstrate sparsity patterns. These patterns arise frequently in several tomographic imaging scenarios, for e.g. in blood vessel tomographic reconstructions \cite{Yang}. The sparsity is incorporated in our model through an $L^2-L^1$ regularization term in our objective functional. An $H^1$ regularization term is also introduced in the functional that helps reducing artifacts. We provide a comprehensive theoretical analysis of our optimization framework. We provide a new proof for the existence of solutions of \eqref{eq: semi-linear equation} with higher regularity, under the assumption that $g\geq 0$, using a fixed point approach. We also prove the existence of minimizers of our minimization problem. We solve the optimization problem using a variable inertial proximal scheme that efficiently handles the non-differentiable $L^1$ regularization term in the objective functional. Finally, we demonstrate the applicability of our reconstruction approach by implementing scheme to several examples.

The article is organized as follows: In Section \ref{sec:minimization problem}, we formulate the minimization problem for the 2P-PAT reconstruction problem. In Section \ref{sec:Theory}, we present some theoretical results about our optimization problem and we also characterize the optimality system. The numerical schemes to solve the forward problem and the optimization problem are discussed in Section \ref{sec:Numerical schemes}. In Section \ref{sec: numerical results}, we present simulation results of our 2P-PAT framework. A section on conclusions completes our work.

\section{A minimization problem}\label{sec:minimization problem}

In this section, we describe the minimization problem corresponding to the 2P-PAT reconstruction problem. We assume $\O$ to be bounded domain in $\Rb^2$.
The authors in \cite{Ren2018} show that, under the assumptions of the boundary function $g \geq 0$, there exists a non-negative solution $u$ of \eqref{eq: semi-linear equation} in $H^1_g(\O)$. Since $g$ represents the density of photons, $g$ is non-negative. Therefore, instead of the photon propagation equation \eqref{eq: semi-linear equation}, we consider the following boundary value problem
\begin{align}\label{eq: modified semi-linear equation}
\begin{array}{rr}
-\nabla\cdot(D(x)\nabla{u(x)})+\sigma(x) u(x) + \mu(x) u^2(x)&=0, \quad\qquad   \mbox{in }\Omega, \\
u(x)&=g(x) \quad \mbox{ on } \partial \Omega
\end{array}
\end{align}
as the model for photon propagation in $\O$. We assume that the diffusion coefficient $D\in W^{1, \infty}(\Omega)$ is known.  Throughout the article, we assume that the absorption coefficients $\sigma$ and $\mu$ belong to the function spaces $L^\sigma_{ad}$ and $L^\mu_{ad}$ respectively, where
 \[
\begin{aligned}
&L^\sigma_{ad} = \lbrace q(x) \in H^1(\Omega): a_\sigma \leq q(x) \leq b_\sigma,~ \forall x\in\Omega,~ a_\sigma,b_\sigma > 0 \rbrace,\\
&L^\mu_{ad} = \lbrace q(x) \in H^1(\Omega): a_\mu \leq q(x) \leq b_\mu,~ \forall x\in\Omega,
~a_\mu,b_\mu > 0  \rbrace.
\end{aligned}
\]
Then the aim is to recover both absorption coefficients $\sigma $ and $\mu$ from the knowledge of two sets boundary illumination functions $g_1,g_2$ and the corresponding initial acoustic wave pressure field $\mathcal{H}_1^{\sigma, \mu},\mathcal{H}_2^{\sigma, \mu}$, where
\begin{align}\label{eq:initial pressure field}
\Hc^{\sigma, \mu}(x) =  \Gamma(x) \left[\sigma(x) u(x) + \mu(x) u^2(x)\right], \quad \mbox{ for } x \in \Omega.
\end{align}
 For a known diffusion coefficient $D$, the equation \eqref{eq: semi-linear equation} can be represented as follows
\begin{equation}\label{eq:eq1}
\begin{aligned}
\mathcal{L}(u, \sigma, \mu, g) =  0.
\end{aligned} 
\end{equation}
We will use an optimization based approach to reconstruct the coefficients  $ \sigma(x)$ and $\mu(x)$. We start by defining the following cost functional
\begin{equation}\label{eq: definition of cost functional}
\begin{aligned}
J(\sigma, \mu, u_1, u_2) &= \sum_{j=1}^2\frac{\alpha_j}{2}\|\mathcal{H}_j^{ \sigma, \mu} - G_j^{\delta}\|^2+\frac{\xi_1}{2}\|\sigma-\sigma_b\|_{H^1(\Omega)}^2 +\frac{\xi_2}{2}\|\mu-\mu_b\|_{H^1(\Omega)}^2 \\
&+\gamma_1\|\sigma-\sigma_b\|_{L^1}  +\gamma_2\|\mu-\mu_b\|_{L^1},
\end{aligned}
\end{equation}
where $u_1,u_2$ satisfy \eqref{eq: semi-linear equation} with boundary source functions $g_1$, $g_2$ respectively, $\sigma_b,\mu_b$ are known background absorption coefficients and $G_j^\delta,~ j=1,2$ are the (possibly noisy) measured initial acoustic wave pressure fields.

We now consider the following constrained minimization problem associated to the above cost functional
\begin{align}\label{eq:minimization problem}
\min_{\sigma, \mu}&J(\sigma, \mu, u_1, u_2),\\
\mbox{s.t.}\ &\mathcal{L}(u_1, \sigma, \mu, g_1)=0,
\\&\mathcal{L}(u_2, \sigma, \mu, g_2)=0. \tag{P}
\end{align}
The first term in the functional \eqref{eq: definition of cost functional} represents a least-square data fitting term for obtaining $\sigma,\mu$ such that $\mathcal{H}_j^{ \sigma, \mu} \approx G_j^\delta,~ j=1,2$. The regularization terms $\|\sigma-\sigma_b\|_{L^1}$ and  $\|\mu-\mu_b\|_{L^1}$ in the above functional \eqref{eq: definition of cost functional}  implement $L^1$ regularization of the minimization problem that helps promote sparsity patterns in the reconstruction of absorption coefficients. The use of such $L^1$ regularization terms has been shown to obtain high contrast in the reconstructions \cite{Roy_CDII, Roy_AET2}. The $H^1$ regularization terms $\|\sigma-\sigma_b\|^2_{H^1}$ and  $\|\mu-\mu_b\|^2_{H^1}$ help in denoising and removal of artifacts, thus, promoting high resolution.

\section{Theory of the minimization problem}\label{sec:Theory}
In this section, we analyze the existence of a solution to the minimization problem \eqref{eq:minimization problem} and, further, characterize this solution through a first-order optimality system. We refer to this minimization problem as the 2P-PAT sparse reconstruction problem (2PPAT-SR). We begin our discussion with the analysis of the solution of \eqref{eq: modified semi-linear equation}. The existence of solution $u \in H^1_g(\O)$ for the boundary value problem \eqref{eq: modified semi-linear equation} has been established in \cite{Ren2018} under the assumptions that the coefficients $D, \sigma, \mu$ are bounded above and below by some positive constants and the boundary function $g$ is the restriction of a continuous function $\varphi \in C^0(\bar{\O})$. The authors also showed the existence of a regular solution $u \in H_g^3(\O)$ under extra assumptions $D, \sigma, \mu$ are in $H^1(\O)$ and $g$ comes from $\varphi \in C^3(\bar{\O})$. Further, the authors show that $u$ is non-negative corresponding to a non-negative boundary function $g$ is non-negative. 

To prove the existence of minimizer of \eqref{eq: definition of cost functional}, we need $u \in H^2(\O)$. For this purpose, we impose weaker assumptions on the coefficients of \eqref{eq: modified semi-linear equation} and boundary function $g$ compared to the assumptions used in \cite{Ren2018}. 
We present a new proof to the existence and uniqueness of solution $u \in H^2(\O)$ for the boundary value problem \eqref{eq: modified semi-linear equation}. We first recall the following well known fixed point theorem, for reference see \cite[Theorem 4, Section 9.2]{Evans_Book}.
\begin{theorem}[Schaefer's Fixed Point Theorem]\label{th:Fixed point Theorem}
Suppose $A : X \longrightarrow X$ is a continuous and compact mapping. Assume further that the set 
$$\{ u \in X : u = \lambda A[u] \mbox{ for some } 0 \leq \lambda  \leq 1\}$$
is bounded. Then $A$ has a fixed point.
\end{theorem}
The following theorem gives the existence and uniqueness of solution $u \in H^2(\O)$ of \eqref{eq: modified semi-linear equation}.
\begin{theorem}\label{th:existence of solution}
Let $\O$ be a bounded domain in $\Rb^2$. Assume  $D(x)\in W^{1,\infty}(\Omega)$, $(\sigma(x),\mu(x))\in L^\sigma_{ad}\times L^\mu_{ad}$ and  $g \in H^{3/2}(\partial \O)$ are given. Then the boundary value problem \eqref{eq: modified semi-linear equation} has a unique solution $u$ in $H^1_0(\O)\cap L^4(\O)$. Further, any weak solution $u$ of \eqref{eq: modified semi-linear equation} is also a strong solution, that is, $u \in H^2(\O)$.
\end{theorem}
\begin{proof}
In order to solve above equation \eqref{eq: modified semi-linear equation}, we start by reducing it to a homogeneous boundary value problem by putting $u = v +\vp$, where $\vp \in H^2(\O)$ is a possible extension of $g$ from boundary $\partial \O$ to whole $\O$. Then, we can verify that the function $v$ satisfies the equation:
\begin{align}\label{eq:semi-linear homogeneous}
-\nabla\cdot(D(x)\nabla{v(x)})+\vartheta(x)  v + \mu(x) v^2 &= f(x), \quad\qquad   \mbox{in }\Omega, \\
v(x) &=0, \qquad \qquad \mbox{ on } \partial \Omega.
\end{align}
where $\vartheta = \sigma +2\mu \vp$ and $f = \nabla\cdot(D(x)\nabla{\vp})-\sigma \vp -\mu \vp^2$. \\

\noindent For a given $v  \in H^1_0(\O) \cap L^4(\O)$, define $$ F(x):= -\mu(x) v^2(x) +f(x).$$ 

\noindent Using conditions on $\vp$, $D$, $\sigma$ and $\mu$ together with $v  \in L^4(\O)$, we see $F \in L^2(\O)$. Hence there exists a unique $w \in H^1_0(\O)$ (dependent on $v$) satisfying  the following linear boundary value problem, see \cite[Chapter 9]{Brezis_Book} and \cite[Chapter 3, Section 7]{Uraltseva_Book}
\begin{align*}
-\nabla\cdot(D(x)\nabla{w(x)})+\vartheta(x) w(x) &= F(x), \quad\qquad   \mbox{in }\Omega, \\
w(x) &=0, \qquad \qquad \mbox{ on } \partial \Omega
\end{align*} 
with the estimate
$$ \|w\|_{H^2(\O)} \leq C \|F\|_{L^2(\O)}$$
for some constant $C$ (dependent only on coefficient functions and the domain $\O$). \vspace{1mm}\\ 
\noindent This motivates us to define the the operator $A :H^1_0(\O) \cap L^4(\O) \rightarrow H^1_0(\O) \cap L^4(\O)$ given by $A[v] =  w$, where $w$ and $v$ are related in  the same manner as above. Further, we have
\begin{align}\label{eq:estimates on Av}
\|A[v]\|_{H^2(\O)} \leq C  \|F\|_{L^2(\O)} \leq C \left(\|v\|_{L^4(\O)} + \|f\|_{L^2(\O)}\right).
\end{align} 
Note that any fixed point of $A$ will solve  \eqref{eq: modified semi-linear equation} which means to obtain a solution of  \eqref{eq: modified semi-linear equation} it is enough to verify the conditions  of Theorem \ref{th:Fixed point Theorem} for $A$, i.e., we need to show that the operator $A$ is continuous, compact and the set $\{ v \in H^1_0(\O) \cap L^4(\O) : v = \lambda A[v] \mbox{ for some } 0 \leq \lambda  \leq 1\}$ is bounded.\vspace{1mm}

\noindent To show continuity of $A$, let us start with a sequence
$$v_k \rightarrow v, \qquad \mbox{ in }\qquad    H^1_0(\O) \cap L^4(\O)$$
then by the inequality \eqref{eq:estimates on Av}, we have $$\sup_k\|w_k\|_{H^2(\O)} < \infty, \quad \mbox{ where }\quad w_k = A[v_k], \mbox{ for } \ k = 1, \dots   $$
Thus there is a subsequence $\{w_{k_j}\}_{j=1}^\infty$ and a function $w \in  H^1_0(\O) \cap L^4(\O)$ with 
$$ w_{k_j}\rightarrow w, \quad \mbox{ in }\quad H^1_0(\O) \cap L^4(\O).$$
Now, 
\begin{align*}
\int_\O \left(D (\nabla w_{k_j} \cdot \nabla \chi) + \vartheta w_{k_j}\chi  \right)dx  = - \int_{\O} \left(\mu  v_{k_j}^2 \chi - f \chi \right) dx, \quad \forall \chi \in H^1_0(\O). 
\end{align*}
Taking the limit $k_j \rightarrow \infty $ we get 
\begin{align*}
\int_\O \left(D (\nabla w \cdot \nabla \chi) + \vartheta w\chi  \right)dx  = - \int_{\O} \left(\mu  v^2 \chi - f \chi \right) dx, \quad \forall \chi \in H^1_0(\O). 
\end{align*}
Hence $ w = A[v]$. This shows the continuity of $A$. The compactness of $A$ also follows by a similar argument, indeed if $\{v_k\}$ is a bounded sequence in $H^1_0(\O) \cap L^4(\O)$, the estimate \eqref{eq:estimates on Av} shows $\{A[v_k]\}_{k=1}^\infty$ is bounded in $H^2(\O)$ and hence possess a strongly convergent subsequence. The only thing remains to prove is the boundedness of the set:
$$ Y = \left\{ v \in H^1_0(\O) \cap L^4(\O) : v = \lambda A[v] \mbox{ for some } 0 \leq \lambda  \leq 1\right\}.$$
Let $v \in H^1_0(\O) \cap L^4(\O)$ such that 
$$v = \lambda A[v], \quad \mbox{ for some } 0 \leq \lambda \leq 1.$$
\noindent Then $v / \lambda=  A[v] \in H^2(\O) \cap H^1_0(\O) \cap L^4(\O)$ and 
$$-\nabla\cdot(D(x)\nabla{v(x)})+\vartheta(x) v(x) = - \lambda\mu v^2 + \lambda f  , \quad \mbox{ a.e.  in }\O. $$
Multiplying the above relation with $v$ and integrating over $\O$ to get  
\begin{align*}
\int_\O D |\nabla v|^2  + \vartheta |v|^2 &= -\int_\O \lambda \mu v^3 dx +   \int_\O \lambda f v dx \\  & \leq  \int_\O f v  dx = \int_\O \left(\frac{1}{\epsilon} f\right) \left(\epsilon v\right) dx, \quad    \mbox{ for any } \epsilon > 0  \\ 
& \leq \frac{\epsilon^2}{2}\int_\O v^2 dx +  \frac{1}{2 \epsilon^2} \int_\O  f^2 dx. 
\end{align*}
This gives 
\begin{align*}
\int_\O D |\nabla v|^2 + \left(\vartheta -\frac{\epsilon^2}{2}\right) |v|^2 
& \leq  \frac{1}{2 \epsilon^2} \int_\O  f^2 dx. 
\end{align*}
Choose an $ \epsilon >0 $ such that $\left(\vartheta -\frac{\epsilon^2}{2}\right)$ is bounded below by positive constant. Using this information together with the fact $D$ is bounded below by a positive constant, we verified that the set $Y$ is bounded. Hence by Schaefer's Theorem \ref{th:Fixed point Theorem}, we conclude that the operator $A$ has a fixed point $v \in H^2(\O) \cap H^1_g(\O) \cap L^4(\O)$. 

To show the uniqueness of the solution $u$, let $u_1$ and $u_2$ be two non-negative solutions of the boundary value problem \eqref{eq: modified semi-linear equation}. Then $w=u_1-u_2$ satisfies the following boundary value problem
\begin{align*}
-\nabla\cdot(D(x)\nabla{w(x)})+\sigma(x) w(x) + \mu(x) w(x)(u_1(x)+u_2(x))&=0, &  \mbox{ in }\Omega, \\
w(x) &=0, &\mbox{ on } \partial \Omega. 
\end{align*}
Multiplying above equation by $w$ and integrating by part , we get
\begin{align*}
\int_{\Omega} D(x)(\nabla{w(x)})^2+\sigma(x) w^2(x) + \mu(x) w^2(x)(u_1(x)+u_2(x))  dx&=0. 
\end{align*}
Since all coefficients are positive and solutions $u_1$, $u_2$ are non-negative therefore the above relation entails $w \equiv 0$. This proves the uniqueness of solution for boundary value problem \eqref{eq: modified semi-linear equation}.
\end{proof}
\begin{remark}\label{rem:well-defined_functional}
The result in Theorem \ref{th:existence of solution} ensures that the initial acoustic wave pressure field $\mathcal{H}^{\sigma,\mu}$ given by \eqref{eq:initial pressure field} belongs to $L^2(\Omega) \cap L^4(\Omega).$ Thus, the functional $J$ given by \eqref{eq: definition of cost functional} is well-defined.
\end{remark}
The solvability of the 2PPAT-SR inversion problem depends on the type of Dirichlet boundary data  $g_j,~j=1,2$. In this context, we have the following lemma from \cite{Ren2018}
\begin{lemma}[Boundary data]\label{boundary_conditions}
Let $g_i,~i=1,2$ be two sets of boundary conditions with $g_i >0$ and $g_1-g_2 >0$. Then $u_1 \neq u_2$ almost everywhere in $\Omega$ and one can uniquely reconstruct $(\sigma,\mu)$ from the two sets of initial acoustic wave pressure fields $\mathcal{H}_i^{ \sigma, \mu},~ i=1,2$.
 \end{lemma}
Next, we state the following lemma about the  Fr\'{e}chet differentiability of the  mapping $u(\sigma, \mu)$ which will be needed later. For proof of this lemma, we refer to \cite[Proposition 2.5]{Ren2018}.
\begin{lemma}\label{differentiable_constraint}
	The map $u(\sigma, \mu)$ defined by \eqref{eq: semi-linear equation} is Fr\'{e}chet differentiable with respect to $\sigma$ and $\mu$ as a mapping from $L^\sigma_{ad}\times L^\mu_{ad}$ to $H^1_g(\Omega)$.
\end{lemma}
 Using Lemma \ref{differentiable_constraint}, we introduce the reduced cost functional
\begin{equation}\label{reduced_func}
\widehat{J}(\sigma,\mu) = J(\sigma,\mu, u_1(\sigma,\mu), u_2(\sigma,\mu)),
\end{equation}
where $u_i(\sigma,\mu)$, $i=1,2$ denotes the unique solution of \eqref{eq:eq1} given $\sigma,\mu$ and $g_i,i=1,2$. The constrained optimization problem (\ref{eq:minimization problem}) can be formulated as an unconstrained one as follows
\begin{equation}\label{reduced_min}
\min_{(\sigma,\mu)\in L^\sigma_{ad}\times L^\mu_{ad}} \hat{J}(\sigma,\mu).
\end{equation}We next investigate the existence of a minimizer to the 2PPAT-SR problem (\ref{eq:minimization problem}).  
\begin{proposition}
Let $g_1,g_2 \in H^{1/2}(\Omega)$. Then there exists a quadruplet $(\sigma^*,\mu^*,u_1^*,u_2^*) \in L^\sigma_{ad}\times L^\mu_{ad} \times H^1_{g_1}(\Omega)\times H^1_{g_2}(\Omega)$ such that $u_i^*, i=1,2$ are solutions to $\mathcal{L}(\sigma,\mu,u_i,g_i)=0, i=1,2$ and $(\sigma^*,\mu^*)$ minimizes $\hat{J}$ in $L^\sigma_{ad}\times L^\mu_{ad} $.
\end{proposition}
\begin{proof}
We observe that $\hat{J}$ is bounded below. This implies there exists a minimizing sequence $(\sigma_m,\mu_m) \in L^\sigma_{ad}\times L^\mu_{ad}$. Since $\hat{J}$ is coercive in $L^\sigma_{ad}\times L^\mu_{ad}$, we have that the sequence $(\sigma_m, \mu_m)$ is bounded. Since $L^\sigma_{ad}\times L^\mu_{ad}$ is a closed subspace of a Hilbert space, it is reflexive. Thus, the sequence $(\sigma_m,\mu_m)$ has a weakly convergent subsequence $(\sigma_{m_l}, \mu_{m_l}) \rightharpoonup (\sigma^*,\mu^*)$. Consequently, the sequences $u_i(\sigma_{m_l},\mu_{m_l}) \rightharpoonup u^*$  in $ H^2(\Omega) \subset H^1_{g_i}(\Omega),~ i=1,2$. Due to the fact that $ H^2(\Omega)$ is compactly embedded in $H^1_{g_i}(\Omega)$, we have $u_i(\sigma_{m_l}, \mu_{m_l}) \rightarrow u^* \in H^1_{g_i}(\Omega)$. Again, since $ H^2(\Omega)$ is compactly embedded in $L^4(\Omega)$, we additionally have $u_i(\sigma_{m_l}, \mu_{m_l}) \rightarrow u^* \in L^4(\Omega)$. We next aim at showing that $u^* = u(\sigma^*,\mu^*) \in H^1_{g_i}(\Omega)$. For this purpose, we consider the weak formulation of the solution of \eqref{eq: semi-linear equation}. The first term in the weak formulation we need to consider is $\langle \sigma_{m_l}u_i(\sigma_{m_l},\mu_{m_l}), \psi \rangle_{L^2(\Omega)}$. By the preceding discussion, we have  $\langle \sigma_{m_l}u_i(\sigma_{m_l}, \mu_{m_l}), \psi \rangle_{L^2(\Omega)}\rightarrow \langle \sigma^*u_i^*, \psi \rangle_{L^2(\Omega)}$. The second term we need to analyze is $\langle \mu_{m_l}u_i^2(\sigma_{m_l},\mu_{m_l}), \psi \rangle_{L^2(\Omega)}$. Since, $\mu_{m_l} \rightharpoonup \mu^*$ in $L^2(\Omega)$ and $u_i(\sigma_{m_l}, \mu_{m_l}) \rightarrow u^* \in L^4(\Omega)$, we have 
$\langle \mu_{m_l}u_i^2(\sigma_{m_l},\mu_{m_l}), \psi \rangle_{L^2(\Omega)} \rightarrow \langle \mu^*(u_i^*)^2, \psi \rangle_{L^2(\Omega)}$.

Thus, $(\sigma^*,\mu^*,u_i^*)$ solves \eqref{eq: semi-linear equation} with boundary condition $g_i$ and by continuity of the map $u(\sigma,\mu)$, we have $u^* = u(\sigma^*,\mu^*)$. Since $\hat{J}$ is sequentially weakly lower semi-continuous, we have that $(\sigma^*,\mu^*,u_1^*,u_2^*) $ minimizes $\hat{J}$ in $L^\sigma_{ad}\times L^\mu_{ad} \times H^1_{g_1}(\Omega)\times H^1_{g_2}(\Omega)$.
\end{proof}

\subsection{Characterization of local minima}
To characterize the solution of our optimization problem through first-order optimality conditions, we write the reduced functional $\hat{J}$ as follows
$$
\hat{J} = \hat{J}_1+\hat{J}_2,~ \hat{J}_i:L_{ad}^\sigma\times L_{ad}^\mu \rightarrow\mathbb{R}^+,~ i=1,2,
$$
 where
\begin{equation}\label{sumoffunctionals}
\begin{aligned}
&\hat{J}_1(\sigma,\mu) = \sum_{j=1}^2\frac{\alpha_j}{2}\|\mathcal{H}_j^{ \sigma, \mu} - G_j^{\delta}\|^2+\frac{\xi_1}{2}\|\sigma-\sigma_b\|_{H^1(\Omega)}^2 +\frac{\xi_2}{2}\|\mu-\mu_b\|_{H^1(\Omega)}^2,\\
&\hat{J}_2(\sigma,\mu) =\gamma_1\|\sigma-\sigma_b\|_{L^1}  +\gamma_2\|\mu-\mu_b\|_{L^1}.
\end{aligned}
\end{equation}

\begin{remark}
The functional $\hat{J}_1$ is smooth and possibly non-convex, while $\hat{J}_2$ is non-smooth and convex.
\end{remark}
\noindent The following property can be proved using arguments in \cite{Kuchment_Steinhauer}.
\begin{proposition}\label{diff_J1}
The reduced functional $\hat{J}_1(\sigma,\mu)$ is weakly lower semi-continuous, bounded below and Fr\'{e}chet differentiable with respect to $\sigma,\mu$.
\end{proposition}
\noindent Next, we are going to define the subdifferential of a non-smooth functional.
\begin{definition}[Subdifferential]
If $\hat{J}$ is finite at a point $(\sigma,\mu)$, the Fréchet subdifferential of $\hat{J}$ at $(\sigma,\mu)$ is defined as follows \cite{ekeland}
\begin{equation}\label{subdifferential}
\partial \hat{J}(\bar{\sigma},\bar{\mu}):=\Bigg\lbrace{\phi\in \left(L^{\sigma}_{ad} \times L_{ad}^{\mu}\right)^*:\liminf_{(\sigma,\mu)\rightarrow (\bar{\sigma},\bar{\mu})}\dfrac{\hat{J}(\sigma,\mu)-\hat{J}(\bar{\sigma},\bar{\mu})-\langle\phi,(\sigma,\mu)-(\bar{\sigma},\bar{\mu})\rangle}{\|(\bar{\sigma},\bar{\mu})-(\sigma,\mu)\|_2}}\geq 0\Bigg\rbrace,
\end{equation}
where $\left(L^{\sigma}_{ad} \times L_{ad}^{\mu}\right)^*$ is the dual space of $L_{ad}^\sigma\times L_{ad}^\mu$. An element $\phi\in \partial \hat{J}(\sigma,\mu)$ is called a subdifferential of $\hat{J}$ at $(\sigma,\mu)$.
\end{definition}
\noindent In our setting, we have the following
\[
\partial\hat{J}(\sigma,\mu) = \nabla_{(\sigma,\mu)} \hat{J}_1(\sigma,\mu)+\partial \hat{J}_2(\sigma,\mu),
\]
since $\hat{J}_1$ is Fr\'{e}chet differentiable by Prop. \ref{diff_J1}. Moreover, for each $\alpha >0$, it holds that 
\[
\partial(\alpha \hat{J}) = \alpha \partial \hat{J}.
\]
The following proposition gives a necessary condition for a local minimum of $\hat{J}$ (see \cite{Roy_AET2}).
\begin{proposition}[Necessary condition]
If $\hat{J}=\hat{J}_1+\hat{J}_2$, with $\hat{J}_1, \hat{J}_2$ given by (\ref{sumoffunctionals}), attains a local minimum at $(\sigma^*,\mu^*)\in L_{ad}^\sigma\times L_{ad}^\mu$, then
\[
\textbf{0}\in \partial \hat{J}(\sigma^*,\mu^*),
\]
or equivalently
\[
-\nabla_{(\sigma,\mu)} \hat{J}_1(\sigma^*,\mu^*)\in \partial \hat{J}_2(\sigma^*,\mu^*).
\]
\end{proposition}
The following variational inequality holds for each $\lambda\in \partial \hat{J}_2(\sigma^*,\mu^*)$ (see \cite{stadler}).
\begin{equation}\label{var_ineq}
\langle\nabla \hat{J}_1(\sigma^*,\mu^*)+\lambda,(\sigma,\mu)-(\sigma^*,\mu^*)\rangle \geq 0,\qquad \forall (\sigma,\mu) \in L_{ad}^\sigma\times L_{ad}^\mu.
\end{equation}
Using the definition of $\hat{J}_2$ in (\ref{sumoffunctionals}) and the fact that $L_{ad}^\sigma\times L_{ad}^\mu$ is reflexive, the inclusion $\lambda\in\partial \hat{J}_2(\sigma^*,\mu^*)$ gives the following characterization of space of $\lambda$
\[
\lambda = (\lambda_1,\lambda_2), \lambda_i \in \Lambda^i_{ad}:=\lbrace \lambda_i \in L^2(\Omega): 0\leq\lambda\leq \gamma_i,\mbox{ a.e. in } \Omega\rbrace, ~ i=1,2.
\]

A pointwise analysis of the variational inequality (\ref{var_ineq}) leads to the existence of a non-negative functions $\lambda_{i,a}^*,\lambda_{i,b}^*\in L^2(\Omega),~ i = 1,2$ that correspond to Lagrange multipliers for the inequality constraints in $L_{ad}^\sigma\times L_{ad}^\mu$. We, thus, have the following first-order optimality system.

\begin{proposition}[First-order necessary conditions]\label{necessary}
The optimal solution of the minimization problem (\ref{reduced_min}) can be characterized by the existence of $(\lambda_1^*,\lambda_2^*,\lambda_{1,a}^*,\lambda_{2,a}^*,\lambda_{1,b}^*,\lambda_{2,b}^*)\in (\Lambda_{ad})^2\times (L^2(\Omega))^4$ such that 
\begin{eqnarray}
&&\label{grad1_sigma}\nabla_\sigma \hat{J}_1(\sigma^*,\mu^*) + \lambda_1^*+\lambda_{1,b}^*-\lambda_{1,a}^*=0,\\
&&\label{grad1_mu}\nabla_\mu \hat{J}_1(\sigma^*,\mu^*) + \lambda_2^*+\lambda_{2,b}^*-\lambda_{2,a}^*=0,\\
&&\label{comp_st_sigma} \lambda_{1,b}^* \geq 0,~ b-\sigma^*\geq 0,~\langle \lambda_{1,b}^*,b-\sigma^* \rangle=0,\\ 
&&\lambda_{1,a}^* \geq 0,~ \sigma^*-a\geq 0,~\langle \lambda_{1,a},\sigma^*-a \rangle=0,\\ 
&&\label{comp_st_mu} \lambda_{2,b}^* \geq 0,~ b-\mu^*\geq 0,~\langle \lambda_{2,b}^*,b-\mu^* \rangle=0,\\ 
&&\lambda_{2,a}^* \geq 0,~ \mu^*-a\geq 0,~\langle \lambda_{2,a},\mu^*-a \rangle=0,\\ 
&&\label{4}\lambda_1^*=\gamma _1\mbox{ a.e. on } \lbrace x\in\Omega:\sigma^*(x) > 0 \rbrace,\\
&&\label{5}\lambda_2^*=\gamma _2\mbox{ a.e. on } \lbrace x\in\Omega:\mu^*(x) > 0 \rbrace,\\
&& \label{comp_end_sigma} 0\leq\lambda_1^*\leq \gamma_1 \mbox{ a.e. on } \lbrace x\in\Omega:\sigma^*(x) = 0 \rbrace,\\
&& \label{comp_end_mu} 0\leq\lambda_2^*\leq \gamma_2 \mbox{ a.e. on } \lbrace x\in\Omega:\mu^*(x) = 0 \rbrace.
\end{eqnarray}
The conditions (\ref{comp_st_sigma})-(\ref{comp_end_mu}) are known as the complementarity conditions for $(\sigma^*,\mu^*,\lambda_1^*,\lambda_2^*)$.
\end{proposition}
\noindent To determine the gradient $\nabla_\sigma \hat{J}_1,\nabla_\mu \hat{J}_1$, we use the adjoint approach (see for e.g., \cite{ RoyAnnunziatoBorzi2016,Roy2018}). This gives the following reduced gradients of $\hat{J}_1$
\begin{equation}\label{L2_gradients}
\begin{aligned}
\nabla_\sigma \hat{J}_1(\sigma^*,\mu^*)=&\alpha_1(\mathcal{H}_1^{ \sigma^*, \mu^*}-G_1^\delta)\Gamma u_1+\alpha_2(\mathcal{H}_2^{ \sigma^*, \mu^*}-G_2^\delta)\Gamma u_2 + u_1v_1+ u_2v_2 + \xi_1 \sigma^*\\
\nabla_\mu \hat{J}_1(\sigma^*,\mu^*)=&\alpha_1(\mathcal{H}_1^{ \sigma^*, \mu^*}-G_1^\delta)\Gamma u_1^2+\alpha_2(\mathcal{H}_2^{ \sigma^*, \mu^*}-G_2^\delta)\Gamma u_2^2 + u_1^2v_1+u_2^2v_2 + \xi_2 \mu^*\\
\end{aligned}
\end{equation}
where $u_1,u_2$ satisfy the forward equations $\mathcal{L}(u_1,\sigma^*,\mu^*,g_1)=0, ~\mathcal{L}(u_2,\sigma^*,\mu^*,g_2)=0$, respectively, and $v_1,v_2$ satisfy
the adjoint equations 
\begin{equation}\label{adj_1}
\begin{aligned}
-\nabla\cdot(D\nabla{v_1})+\sigma^* v_1 + 2\mu^*u_1 v_1&=-\alpha_1\Gamma (\sigma^* u_1+\mu^*u_1^2-G_1^\delta)\cdot (\sigma^* + 2u_1)~ \mbox{in }\Omega, \\
v_1&=0, \quad \mbox{ on } \partial \Omega
\end{aligned}
\end{equation}
\begin{equation}\label{adj_2}
\begin{aligned}
-\nabla\cdot(D\nabla{v_2})+\sigma^* v_2 + 2\mu^*u_2 v_2&=-\alpha_2\Gamma (\sigma^* u_2+\mu^*u_2^2-G_2^\delta)\cdot (\sigma^* + 2|u_2|)~ \mbox{in }\Omega, \\
v_2&=0, \quad \mbox{ on } \partial \Omega.
\end{aligned}
\end{equation}
The complementarity conditions (\ref{comp_st_sigma})-(\ref{comp_end_mu}) can be rewritten in a compact form as follows. Define 
\begin{equation}\label{eq:c}
\begin{aligned}
&c_1^* = \lambda_1^*+\lambda_{1,b}^*-\lambda_{1,a}^*,\\
&c_2^* = \lambda_2^*+\lambda_{2,b}^*-\lambda_{2,a}^*.
\end{aligned}
\end{equation}
Then the triplets $(\lambda_1^*,\lambda_{1,a}^*,\lambda_{1,b}^*),(\lambda_2^*,\lambda_{2,a}^*,\lambda_{2,b}^*)$ are obtained by solving the following equations
\begin{equation}\label{lambda_comp}
\begin{aligned}
&\lambda_i^* = \min(\gamma_i,\max(0,c_i^*)),\\
&\lambda_{i,a}^* = -\min(0,c_i^*+\gamma_i),\\
&\lambda_{i,b}^* = \max(0,c_i^*-\gamma_i),
\end{aligned}
\end{equation}
for $i=1,2$ (see \cite{stadler}).
For each $k\in\R^+$, define the following quantity
$$
\begin{aligned}
E_1(\sigma^*,c_1^*) = \sigma^*&-\max\lbrace 0,\sigma^*+k(c_1^*-\gamma_1)\rbrace+\max\lbrace 0,\sigma^*-b+k(c_1^*-\gamma_1)\rbrace\\
&-\min\lbrace 0,\sigma^*+k(c_1^*+\gamma_1)\rbrace+\min\lbrace 0,\sigma^*-a+k(c_1^*+\gamma_1)\rbrace.\\
E_2(\mu^*,c_2^*) = \mu^*&-\max\lbrace 0,\mu^*+k(c_2^*-\gamma_2)\rbrace+\max\lbrace 0,\mu^*-b+k(c_2^*-\gamma_2)\rbrace\\
&-\min\lbrace 0,\mu^*+k(c_2^*+\gamma_2)\rbrace+\min\lbrace 0,\mu^*-a+k(c_2^*+\gamma_2)\rbrace.\\
\end{aligned}
$$
The following lemma determines the complementarity conditions (\ref{comp_st_sigma})-(\ref{comp_end_mu}) in terms of $E_1,E_2$ (see \cite[Lemma 2.2]{stadler}).
\begin{lemma}\label{Comp_conditions}
The complementarity conditions (\ref{comp_st_sigma})-(\ref{comp_end_mu}) are equivalent to the following
\begin{equation}\label{comp_equality}
E_1(\sigma^*,c_1^*)=0 = E_2(\mu^*,c_2^*),
\end{equation}
where $c_i,~i=1,2$ are defined in (\ref{eq:c}). 
\end{lemma}
\noindent Using the gradients in (\ref{L2_gradients}) and Lemma \ref{Comp_conditions}, the optimality conditions (\ref{adj_1})-(\ref{comp_end_mu}) for the 2PPAT-SR problem can be rewritten as follows 
\begin{proposition}\label{nec_opt}
A local minimizer $(u_1, u_2,\sigma^*,\mu^*)$ of the problem (\ref{eq:minimization problem}) can be characterized by the existence of $(v_1,v_2,c_1^*,c_2^*)\in H^1_0(\Omega)\times H^1_0(\Omega)\times L_{ad}^\sigma\times L_{ad}^\mu$, such that the following system is satisfied
\begin{equation}\label{eq:optimality}
\begin{aligned}
-\nabla\cdot(D\nabla{u_1})+\sigma^* u_1 + \mu^* u_1^2=0, \qquad   \mbox{in }\Omega,& \\
u_1=g_1, \quad \mbox{ on } \partial \Omega,&\\
-\nabla\cdot(D\nabla{v_1})+\sigma^* v_1 + 2\mu^*u_1 v_1=-\alpha_1\Gamma (\sigma^* u_1+\mu^*u_1^2-G_1^\delta)\cdot (\sigma^* + 2u_1)~ \mbox{in }\Omega,& \\
v_1=0, \quad \mbox{ on } \partial \Omega,&\\
-\nabla\cdot(D\nabla{u_2})+\sigma^* u_2 + \mu^* u_2^2=0, \qquad   \mbox{in }\Omega, \\
u_2=g_2, \quad \mbox{ on } \partial \Omega,&\\
-\nabla\cdot(D\nabla{v_2})+\sigma^* v_2 + 2\mu^*u_2 v_2=-\alpha_2\Gamma (\sigma^* u_2+\mu^*u_2^2-G_2^\delta)\cdot (\sigma^* + 2u_2)~ \mbox{in }\Omega,& \\
v_2=0, \quad \mbox{ on } \partial \Omega,&\\
\alpha_1(\mathcal{H}_1^{ \sigma^*, \mu^*}-G_1^\delta)\Gamma u_1+\alpha_2(\mathcal{H}_2^{ \sigma^*, \mu^*}-G_2^\delta)\Gamma u_2 + u_1v_1+ u_2v_2 + \xi_1 \sigma^*=0,&\\
\alpha_1(\mathcal{H}_1^{ \sigma^*, \mu^*}-G_1^\delta)\Gamma u_1^2+\alpha_2(\mathcal{H}_2^{ \sigma^*, \mu^*}-G_2^\delta)\Gamma u_2^2 + u_1^2v_1+ u_2^2v_2 + \xi_2 \mu^*=0,&\\
E_1(\sigma^*,c_1^*)=0 ,&\\
E_2(\mu^*,c_2^*)=0.\\
\end{aligned}
\end{equation}
\end{proposition}
\section{Numerical schemes for solving the 2PPAT-SR inverse problem}\label{sec:Numerical schemes}
\subsection{Picard type method to solve the forward problem}
In this section we propose a Picard type iterative scheme to solve the semi-linear boundary value problem \eqref{eq: modified semi-linear equation}. The algorithm is given as follows
\begin{algorithm}[Picard-type algorithm]\label{Picard_algo}\
\begin{enumerate}
\item  \textbf{\textup{Input}}: Initial guess $u_0$, $D$, $\sigma$, $\mu$, $g$, $N$ and $TOL$\\
   Initialize: $err_0=1$, $k=0$
\item \textbf{\textup{While}} { $err_{k}>TOL$ and $k<N$} do
\item Solve the following linear elliptic boundary value problem
\begin{align*}
-\nabla\cdot(D(x)\nabla{u_{k+1}(x)})+\sigma(x) u_{k+1}(x) + \mu(x) u_{k}(x)u_{k+1}(x)&=0, \quad\qquad   \mbox{in }\Omega, \\
u_{k+1}(x)&=g(x), \quad \mbox{ on } \partial \Omega
\end{align*}
to get $u_{k+1}$ for $k \geq 0$
\item  $err_{k+1} = \|u_{k+1}-u_k\|_2$
 \item $k=k+1$
\item end
\end{enumerate}
\end{algorithm}
We now show the convergence of the Picard algorithm \ref{Picard_algo} to the solution of \eqref{eq: semi-linear equation}.
\begin{theorem}
Let $D, \sigma,\mu$ be non-negative functions in $L^\infty(\O)$ and $g$ be non-negative function in $C^0(\partial \O)$. Then the iterative sequence $\{u_k\}$, we obtained from the above Picard's method, converges in $H^1(\O)$ and the limit $u$ is a solution of the following  semi-linear elliptic boundary value problem
\begin{equation*}
\begin{aligned}
-\nabla\cdot(D(x)\nabla{u(x)})+\sigma(x) u(x) + \mu(x) u^2(x)&=0, \quad\qquad   \mbox{in }\Omega, \\
u(x)&=g(x), \quad \mbox{ on } \partial \Omega.
\end{aligned}
\end{equation*}
\end{theorem}
\begin{proof}By completeness of $H^1(\O)$ to show the convergence of  sequence $\{u_k\}$ in $H^1(\O)$, we only need to show that the sequence $\{u_k\}$ is a Cauchy sequence in $H^1(\O)$.  To achieve this goal, we will show the following contraction type relation for any $k \geq 1$
\begin{align*}
\|u_{k+1}-u_k\|_{H^1(\O)}\leq \g \|u_{k}-u_{k-1}\|_{H^1(\O)}\leq \cdots \leq \g^k \|u_{1}-u_0\|_{H^1(\O)}, \quad \mbox{ for some } \g < 1.
\end{align*}
We start with $u_2$ and $u_1$, recall from above Picard's type algorithm \ref{Picard_algo} that the iterates $u_1$ and $u_2$ satisfy the following two BVP's respectively
\begin{align}
\begin{array}{rl}
-\nabla\cdot(D(x)\nabla{u_1}(x))+\sigma(x) u_1(x) + \mu(x) u_0(x)u_1(x)&=0, \quad\qquad   \mbox{in }\Omega, \\
u_1(x)&=g(x), \quad \mbox{ on } \partial \Omega.
\end{array} \label{eq: equation for u1}\\
\begin{array}{rl}
-\nabla\cdot(D(x)\nabla{u_1}(x))+\sigma(x) u_2(x) + \mu(x) u_1(x)u_2(x)&=0, \quad\qquad   \mbox{in }\Omega, \\
u_2(x)&=g(x), \quad \mbox{ on } \partial \Omega.
\end{array}\label{eq: equation for u2}
\end{align}
Then by direct substitution, we see that the difference  $\bar{u} = u_2 - u_1$  solves  
\begin{align*}
-\nabla\cdot(D(x)\nabla{\bar{u}}(x))+\sigma(x) \bar{u}(x) + \mu(x)\bar{u}&=\mu u_2(u_0 -u_1), \quad\qquad   \mbox{in }\Omega, \\
\bar{u}(x)&=0, \quad \quad\qquad  \quad\qquad  \mbox{ on } \partial \Omega.
\end{align*}
With the help of regularity estimates for elliptic boundary value problem, we get 
$$ \|\bar{u}\|_{H^1(\O)} \leq  \|\bar{u}\|_{H^2(\O)} \leq C \|\mu u_2(u_0 -u_1)\|_{L^2(\O)}.$$ 
Consider the right hand side of the above inequality
\begin{align*}
 \|\mu u_2(u_0 -u_1)\|_{L^2(\O)}  &=  \left(\int_\O|\mu|^2 |u_2|^2|u_0 -u_1|^2 dx\right)^{\frac{1}{2}}\\
 &\leq \underbrace{\| \mu\|_{L^\infty}\| g\|_{L^\infty}}_{\widetilde{C}} \|(u_0 -u_1)\|_{L^2(\O)}. 
\end{align*}
Using this inequality, we have 
$$ \|u_2-u_1\|_{H^1(\O)} \leq  \|u_2 -u_1\|_{H^2(\O)} \leq\underbrace{ C \widetilde{C}}_{\g} \|(u_0 -u_1)\|_{L^2(\O)}\leq \g \|(u_0 -u_1)\|_{H^1(\O)}.$$ 
By exactly same argument, we get
\begin{align*}
\|u_{k+1}-u_k\|_{H^1(\O)}\leq \g \|u_{k}-u_{k-1}\|_{H^1(\O)}.
\end{align*}
Thus, we have the required relation
$$ \|u_{k+1}-u_k\|_{H^1(\O)}\leq \g^k \|(u_0 -u_1)\|_{H^1(\O)}.$$ 
We can make $\g<1$ by choosing appropriate $g$ and $\mu$.  Hence, the sequence $\{u_k\}$ is a Cauchy sequence in $H^1(\O)$ and hence converges to a limit $u$ in $H^1(\O)$.
To complete the proof of our theorem, the only thing remain to show is that $u$ solve  
\begin{equation*}
\begin{aligned}
-\nabla\cdot(D(x)\nabla{u(x)})+\sigma(x) u(x) + \mu(x) u^2(x)&=0, \quad\qquad   \mbox{in }\Omega, \\
u(x)&=g(x), \quad \mbox{ on } \partial \Omega.
\end{aligned}
\end{equation*}
We know each $u_k$ satisfies
$$ \int_\O D \nabla u_k\cdot  \nabla \vp dx + \int_\O \sigma u_k \vp dx+\int_\O \mu u_{k-1}u_k \vp dx = 0, \quad \mbox{ for all } \quad \vp \in C^\infty_0(\O).$$
The  convergence of $u_{k}\longrightarrow  u$ in $H^1(\O)$ implies the convergence $ \nabla u_k \longrightarrow\nabla u$ in $L^2(\O)$ and the convergence $\sigma u_{k}\longrightarrow  \sigma u$ in $H^1(\O)$ as $k \rightarrow \infty$. Additionally, the strong convergence of $\{u_k\}$ in $H^1(\O)$ will guarantee the weak convergence of $u_{k-1}u_k  \rightharpoonup  u^2$ in $L^2(\O)$.  Thus we have 
$$ \int_\O D \nabla u_k\cdot  \nabla \vp dx + \int_\O \sigma u_k \vp dx+\int_\O \mu u_{k-1}u_k \vp dx \longrightarrow \int_\O D \nabla u\cdot  \nabla \vp dx + \int_\O \sigma u \vp dx+\int_\O \mu u^2 \vp dx $$
for all $\vp \in C^\infty_0(\O)$. 
Therefore 
$$\int_\O D \nabla u\cdot  \nabla \vp dx + \int_\O \sigma u \vp dx+\int_\O \mu u^2 \vp dx = 0, \quad \mbox{ for all } \quad \vp \in C^\infty_0(\O). $$
This completes the proof of the theorem. 
\end{proof}

\subsection{Variable inertial proximal method for solving the optimality system}
For solving the optimality system \eqref{eq:optimality}, we use a class of iterative schemes known as the proximal method. The fundamental idea behind a proximal scheme is to minimize an upper bound of the objective function $\hat{J}$, instead of directly minimizing the functional. This is done using a proximal operator that involves a gradient update of the minimizer. The upper bound is given in terms of the Lipschitz constant  $L$ for the gradient of the functional $\hat{J}_1$. In a special type of proximal method, the exact value of $L$ is not computed directly. Instead an upper bound for $L$ is computed at each iterative step that leads to a fixed step size in the gradient update, known as the inertial parameter. The resulting scheme is known as the variable inertial proximal method (VIP) \cite{Roy_CDII} and has nice convergent properties. We summarize the VIP scheme in the algorithm below as given in \cite{Roy_CDII}

\begin{algorithm}[Variable inertial proximal (VIP) method]\label{proximal_algo}\
\begin{enumerate}

\item  Input: $\beta$, $\hat{J}_1$, $\sigma_0=\sigma_{-1}$, $\mu_0=\mu_{-1}$, $TOL$, $n>1$, $L_0>0$\\
   \textbf{Initialize:} $E_1^0=E_2^0=1$, $k=0$, choose $\theta\in(0,1)$ and $c_1 < 2$ and $c_2>0$;
\item While {$\|E_1^{k-1}\|+\|E_2^{k-1}\|>TOL$} do
\item Compute $\nabla_\sigma \hat{J}_1(\sigma_{k},\mu_k)$, $\nabla_\mu \hat{J}_1(\sigma_{k},\mu_k)$ 
\item Backtracking: Find the smallest non-negative integer $i$ such that with \\
\hspace*{3ex}$\tilde{L}=n^{i}L_{k-1}$
\begin{align*}
\hat{J}_1(\tilde{\sigma},\tilde{\mu})&\leq\hat{J}_1(\sigma_{k},\mu_k)+\left< \nabla_\sigma \hat{J}_1(\sigma_{k},\mu_k),\tilde{\sigma}-\sigma_{k}\right>+ \left< \nabla_\mu \hat{J}_1(\sigma_{k},\mu_k),\tilde{\mu}-\mu_{k}\right>\\
&+\frac {\tilde{L}}{2} \left(\|\tilde{\sigma}-\sigma_{k}\|^2+\|\tilde{\mu}-\mu_{k}\|^2 \right)
\end{align*}
\hspace*{3ex}
where $\tilde{\sigma}=\mathbb{S}^{L_{ad}^\sigma}_{\gamma\, s}\left(\sigma_{k}-s\, (\nabla_\sigma \hat{J}_1)_{H^1}(\sigma_{k},\mu_k)+\theta(\sigma_k-\sigma_{k-1})\right)$\\
\hspace*{10ex}$\tilde{\mu}=\mathbb{S}^{L_{ad}^\mu}_{\gamma\, s}\left(\sigma_{k}-s\, (\nabla_\mu \hat{J}_1)_{H^1}(\sigma_k,\mu_{k})+\theta(\mu_k-\mu_{k-1})\right)$,\\
\hspace*{10ex} $s = c_1 (1-\theta)/(\tilde{L}+2c_2)$,
\item  Set $L_k=\tilde{L}$ and $s_k=c_1 (1-\theta)/(L_k+2c_2)$
\item  $\sigma_{k+1}=\mathbb{S}^{L_{ad}\sigma}_{\gamma\, s_k}\left(\sigma_{k}-s_k \, (\nabla_\sigma \hat{J}_1)_{H^1}(\sigma_{k},\mu_k)+\theta(\sigma_k-\sigma_{k-1})\right)$\\
$\mu_{k+1}=\mathbb{S}^{L_{ad}^\mu}_{\gamma\, s_k}\left(\mu_{k}-s_k \, (\nabla_\mu \hat{J}_1)_{H^1}(\sigma_k,\mu_{k})+\theta(\mu_k-\mu_{k-1})\right)$
 \item  $c_1^k=-(\nabla_\sigma \hat{J}_1)_{H^1}(\sigma_{k},\mu_k)$, $c_2^k=-(\nabla_\mu \hat{J}_1)_{H^1}(\sigma_k,\mu_{k})$
 \item  $E_1^{k}=E(\sigma_k,c_1^k)$, $E_2^{k}=E(\mu_k,c_2^k)$
 \item $k=k+1$   
\item end
\end{enumerate}
\end{algorithm}

\section{Numerical results}\label{sec: numerical results}

We first demonstrate the convergence of the Picard scheme given in Algorithm \ref{Picard_algo} for solving \eqref{eq: semi-linear equation}. We use the method of manufactured solutions to construct an exact solution for  \eqref{eq: semi-linear equation} with a non-zero source term $f(x_1,x_2)$ on the right hand side. We set $D(x_1,x_2) = 1.0, ~\sigma(x_1,x_2)= \sin(x_1)\sin(x_2), \mu=1$. Further, we choose $\Omega = (0,1)\times(0,1)$. The boundary condition is given as $g(x_1,x_2) = \sin(x_1)\sin(x_2)$ and the right-hand side $f(x_1,x_2)= 2\sin(x_1)\sin(x_2) + 2(\sin(x_1)\sin(x_2))^2.$ With the preceding choices of the parameters, the exact solution is given as $u_{ex} = \sin(x_1)\sin(x_2)$. The solution error is evaluated based on the following discrete $L^1$ norm
 $$
 \|u\|_1=h^2 \sum_{i,j=0}^{N_x}|u_{i,j}|,
 $$
 which we identify with $L^1_h$.
The discrete $L^1$ error is defined as follows
$$
Err= \|u-u_{ex}\|_1.
$$
Table \ref{table:conv} shows the results of experiments that demonstrate the convergence of the Picard algorithm. We see that the resulting order of convergence is $\mathcal{O}(h)$.

\begin{table}[H]
\centering
\begin{tabular}{|c| c| c| c|}
\hline
$N_x$ & $Err$ & Order\\ [0.5ex]
\hline
25  & 1.70e-3  & -- \\
50  & 8.77e-4 & 0.96  \\
100  & 4.39e-4 & 0.99 \\
200  & 2.20e-4 & 1.00 \\[1ex]
\hline
\end{tabular}
\caption{Convergence of the Picard algorithm given in Algorithm \ref{Picard_algo}}
\label{table:conv}
\end{table} 
 
We now present the results of numerical experiments obtained using the VIP scheme to solve the 2PPAT-SR reconstruction problem. We choose our domain in the experiments below as $\Omega = (-1,1)\times(-1,1)$. We discretize $\Omega$ into 150 equally spaced points in both $x$ and $y$ directions. The boundary illuminations for solving \eqref{eq: semi-linear equation} to generate two sets of initial acoustic wave pressure field data are chosen as $g_1(x,y) = 1.0, ~g_2(x,y) = 2.0$. Such a choice of boundary conditions are consistent with Lemma \ref{boundary_conditions} that ensure unique solvability of the 2PPAT-SR reconstruction problem. The background values $\sigma_b$ and $\mu_b$ are chosen to be 0.1 and 0.01 respectively, unless otherwise mentioned and $D$ is chosen to be 0.1$\sigma$ while generating the data with a known $\sigma$. The weights of the functional $J$ given in \eqref{eq: definition of cost functional} are chosen as $\alpha_1=\alpha_2=1,\xi_1=0.01,\xi_2=0.01,\gamma_1=0.1,\gamma_2=0.1.$ The value of the Gr\"{u}neisen coefficient is chosen to be 1.0. To generate the data $G_i^\delta,~ i=1,2$, we first solve for $u_i$ in (\ref{eq: semi-linear equation}) with given test values of $\sigma,\mu$ and boundary illumination data $g_i$ on a finer mesh with $N=400$ using the Picard iterative scheme given in Algorithm \ref{Picard_algo}. We then compute $G_i^\delta$ on the finer mesh using the values of $\sigma,\mu,u_i$ from \eqref{eq:initial pressure field}. Finally, we restrict $G_i^\delta$ onto the coarser mesh with $N=150$ and use this as our given data.

In test case 1,  we consider a phantom represented by a disk centered at $(0.25,0.25)$ and having radius 0.25. The value of $\sigma$ inside the disk is 1 and outside is 0. The corresponding value of $\mu$ inside the disk is $0.1$ and outside is 0. The plots of the actual phantoms for $\sigma$ and $\mu$ are shown in Figure \ref{disk}.

\begin{figure}[H]
\centering
\subfloat[Exact $\sigma$]{\includegraphics[width=0.3\textwidth, height=0.25\textwidth]{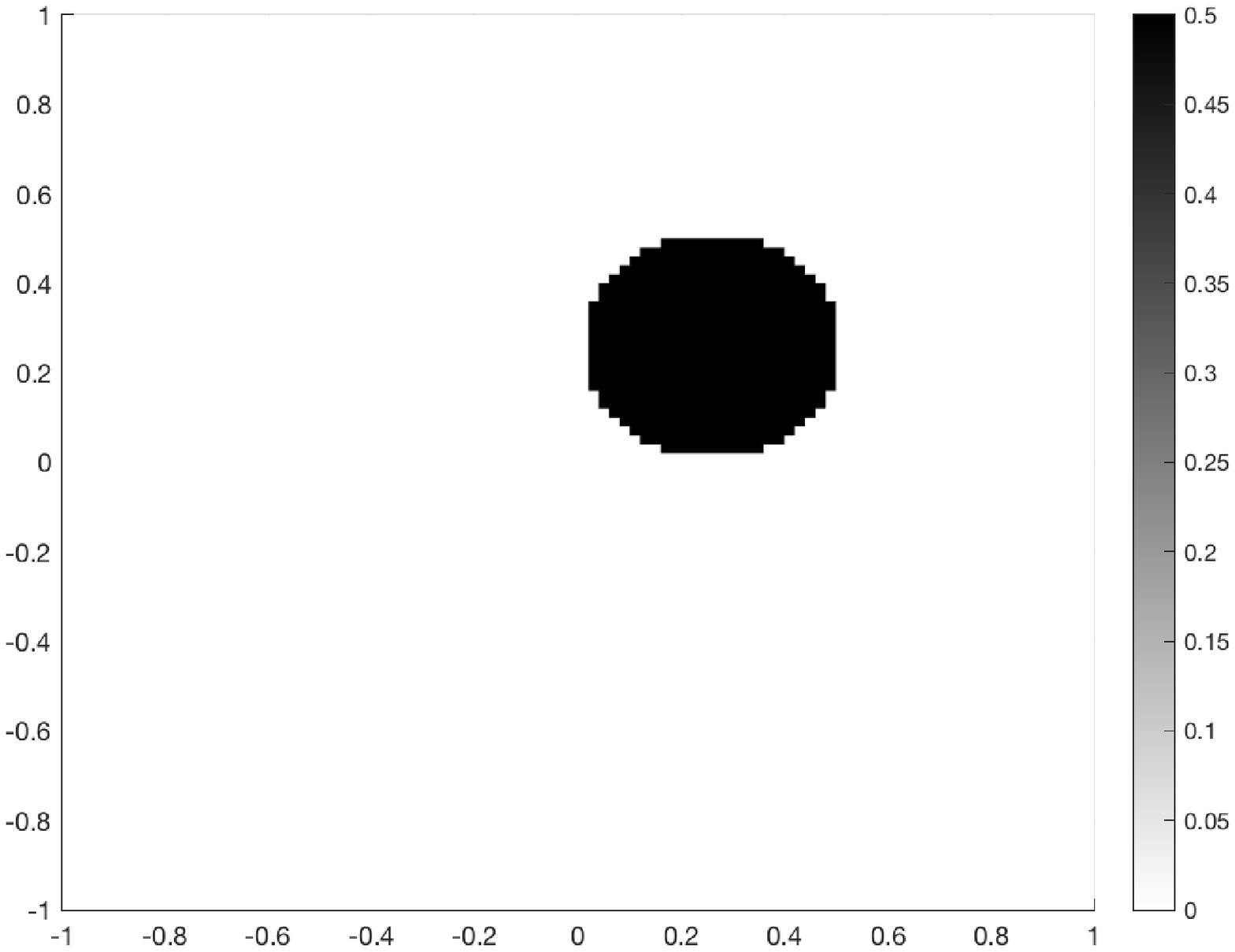}\label{sigma_disk_exact}}\hspace{10mm}
\subfloat[Reconstructed $\sigma$]{\includegraphics[width=0.3\textwidth, height=0.25\textwidth]{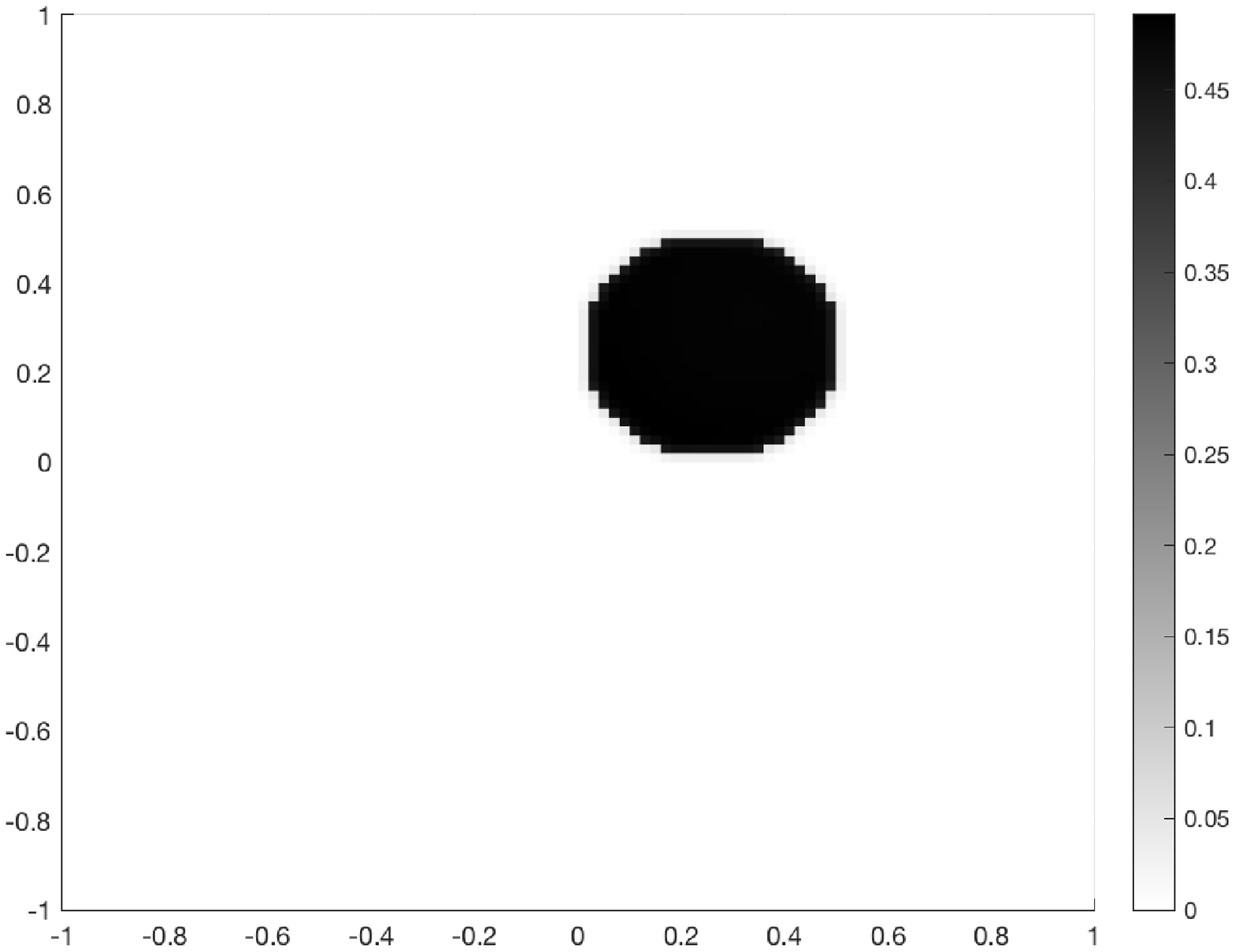}\label{sigma_disk_recon}}\hspace{10mm}\\
\subfloat[Exact $\mu$]{\includegraphics[width=0.3\textwidth, height=0.25\textwidth]{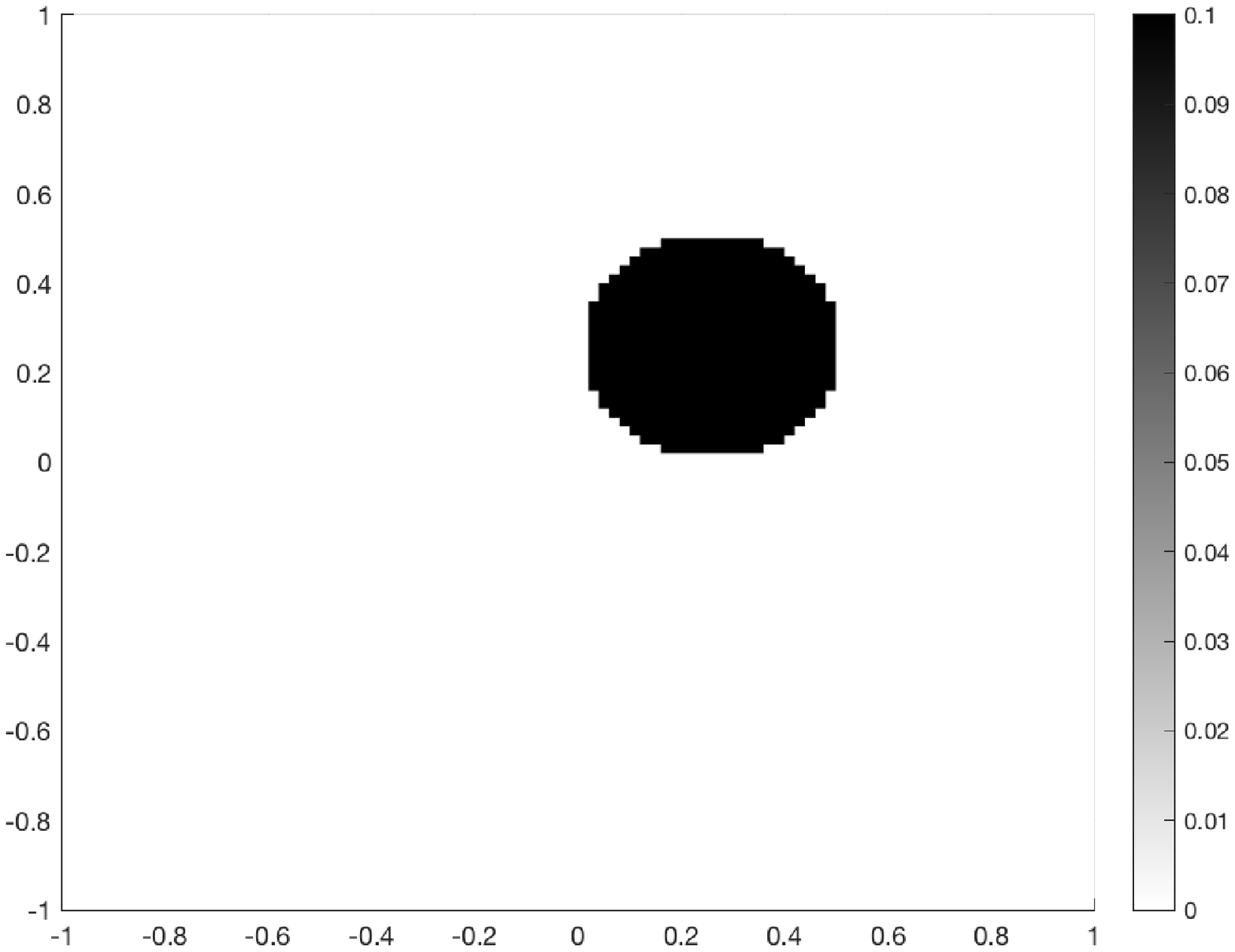}\label{mu_disk_exact}}\hspace{10mm}
\subfloat[Reconstructed $\mu$]{\includegraphics[width=0.3\textwidth, height=0.25\textwidth]{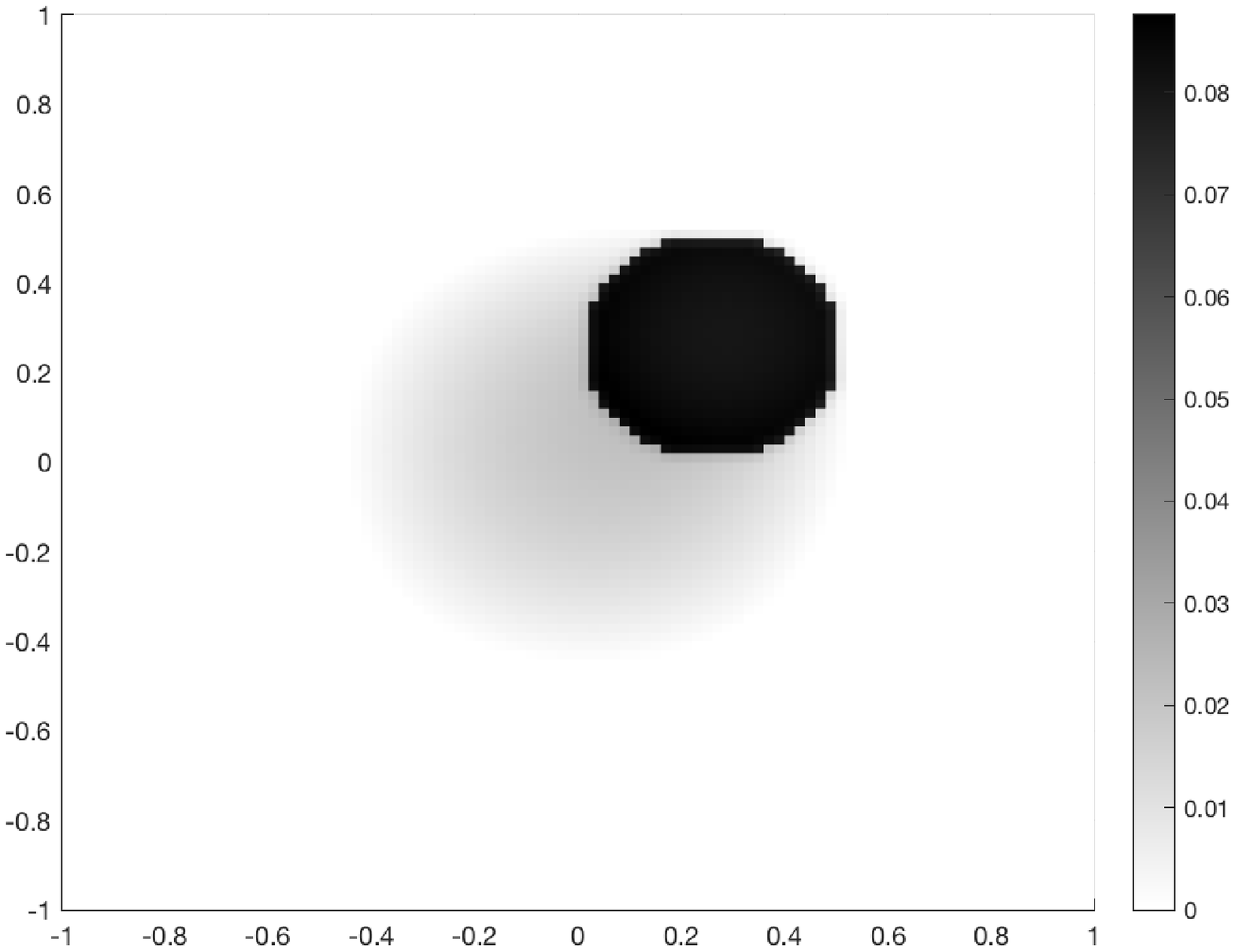}\label{mu_disk_recon}}\hspace{10mm}\\
 \caption{Test Case 1-Reconstructions of the disk phantom with the 2PPAT-SR framework}
    \label{disk}
  \end{figure}
From Figure \ref{sigma_disk_recon} and \ref{mu_disk_recon}, we see that the reconstructions of both $\sigma$ and $\mu$ are of high resolution and high contrast. The value small shaded region around the disk in the reconstruction of $\mu$ is close to 0.02 and, thus, we only encounter a miniscule loss of contrast. 

In test case 2, we consider a heart lung phantom for both $\sigma$ and $\mu$. For $\sigma$, the background value of the phantom is 0 that is perturbed into two ellipses that represent the lungs with value 1 and into a disk representing heart with value 0.5. The value of $\mu$ inside the ellipses and the disk is computed as $\mu=0.1\sigma$. The plots of the exact and the reconstructed phantoms are shown in Figure \ref{heart}.

\begin{figure}[H]
\centering
\subfloat[Exact $\sigma$]{\includegraphics[width=0.3\textwidth, height=0.25\textwidth]{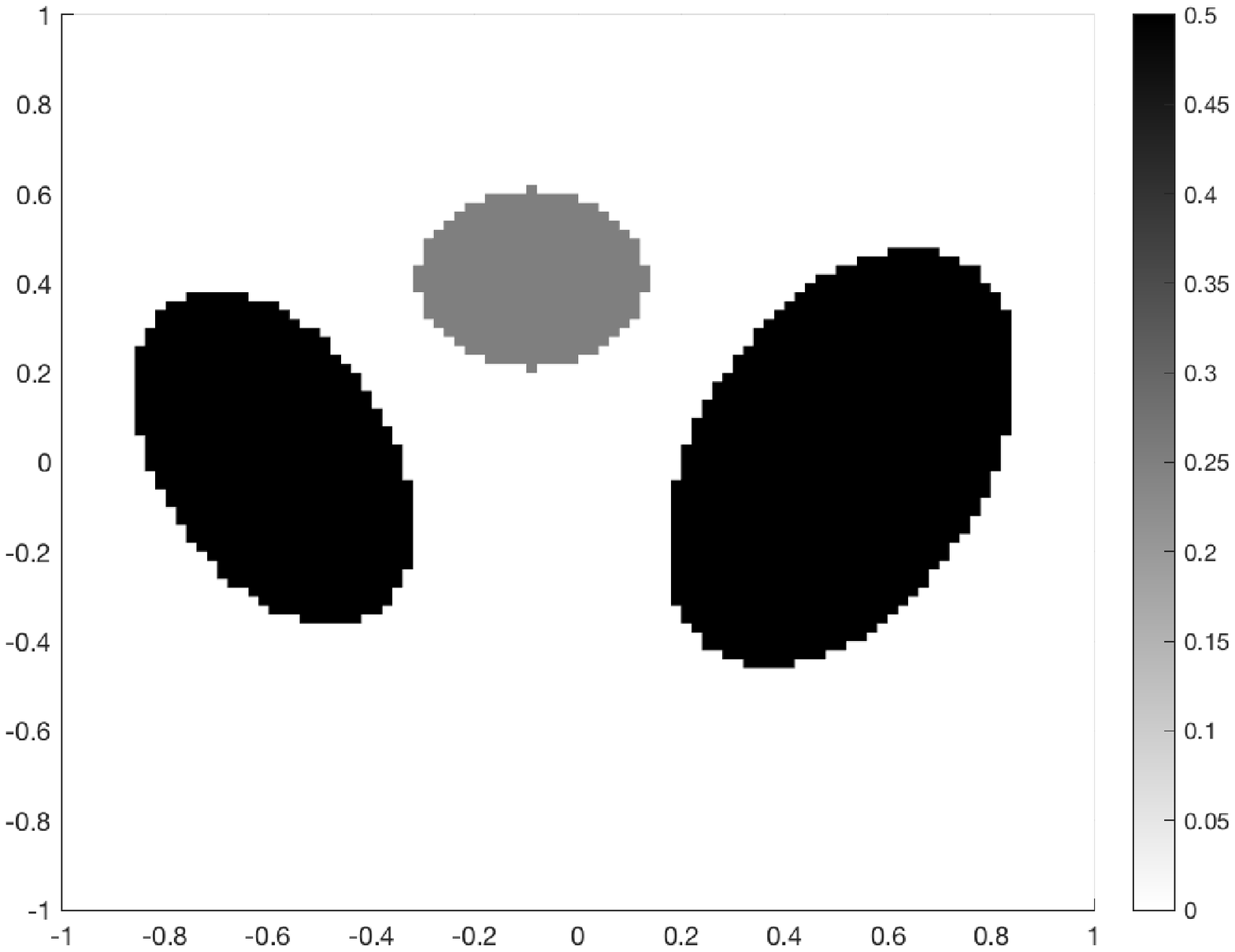}\label{sigma_heart_exact}}
\subfloat[Reconstructed $\sigma$]{\includegraphics[width=0.3\textwidth, height=0.25\textwidth]{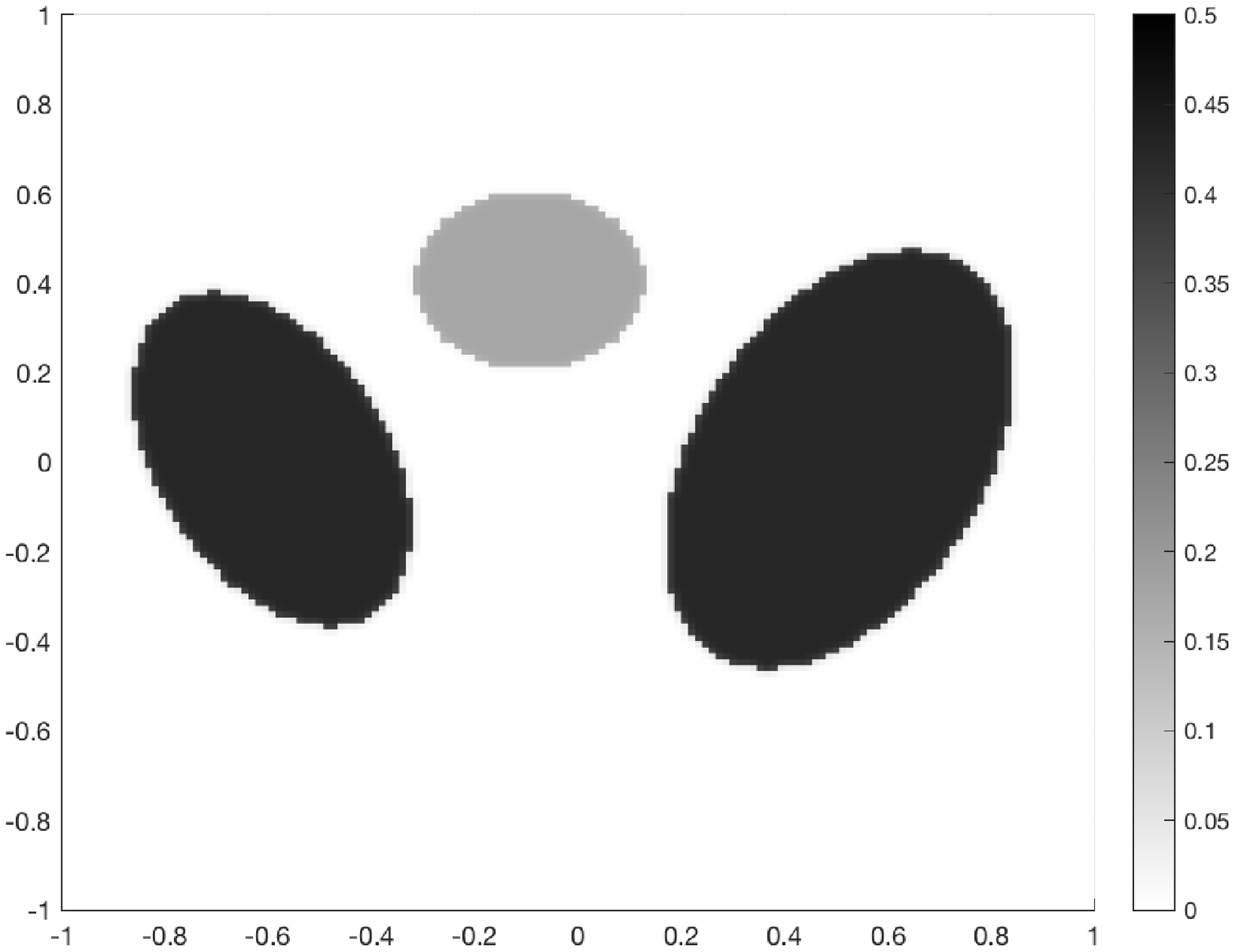}\label{sigma_heart_recon}}
\subfloat[Reconstructed $\sigma$ with 20\%]{\includegraphics[width=0.3\textwidth, height=0.25\textwidth]{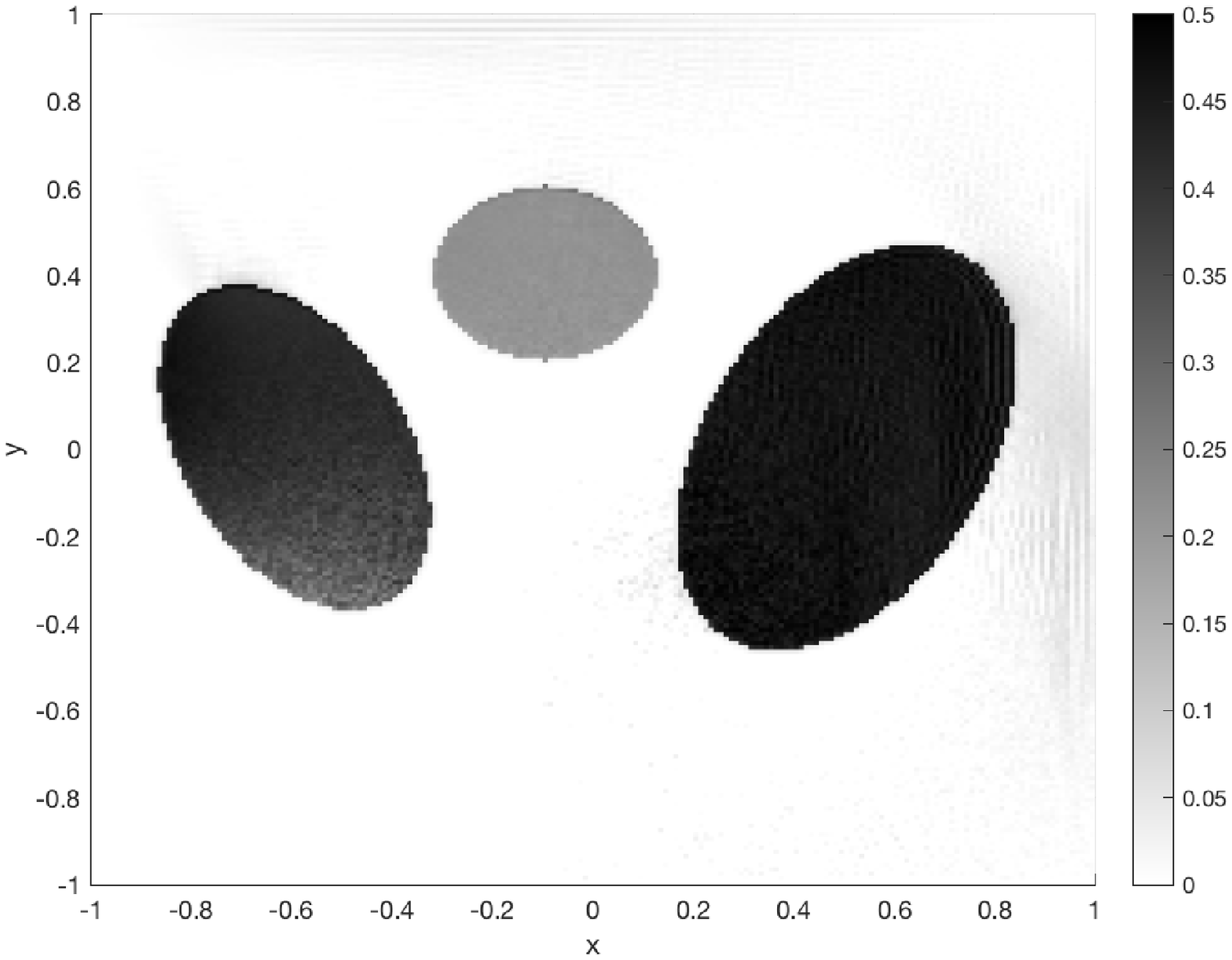}\label{sigma_heart_recon_20}}\\
\subfloat[Exact $\mu$]{\includegraphics[width=0.3\textwidth, height=0.25\textwidth]{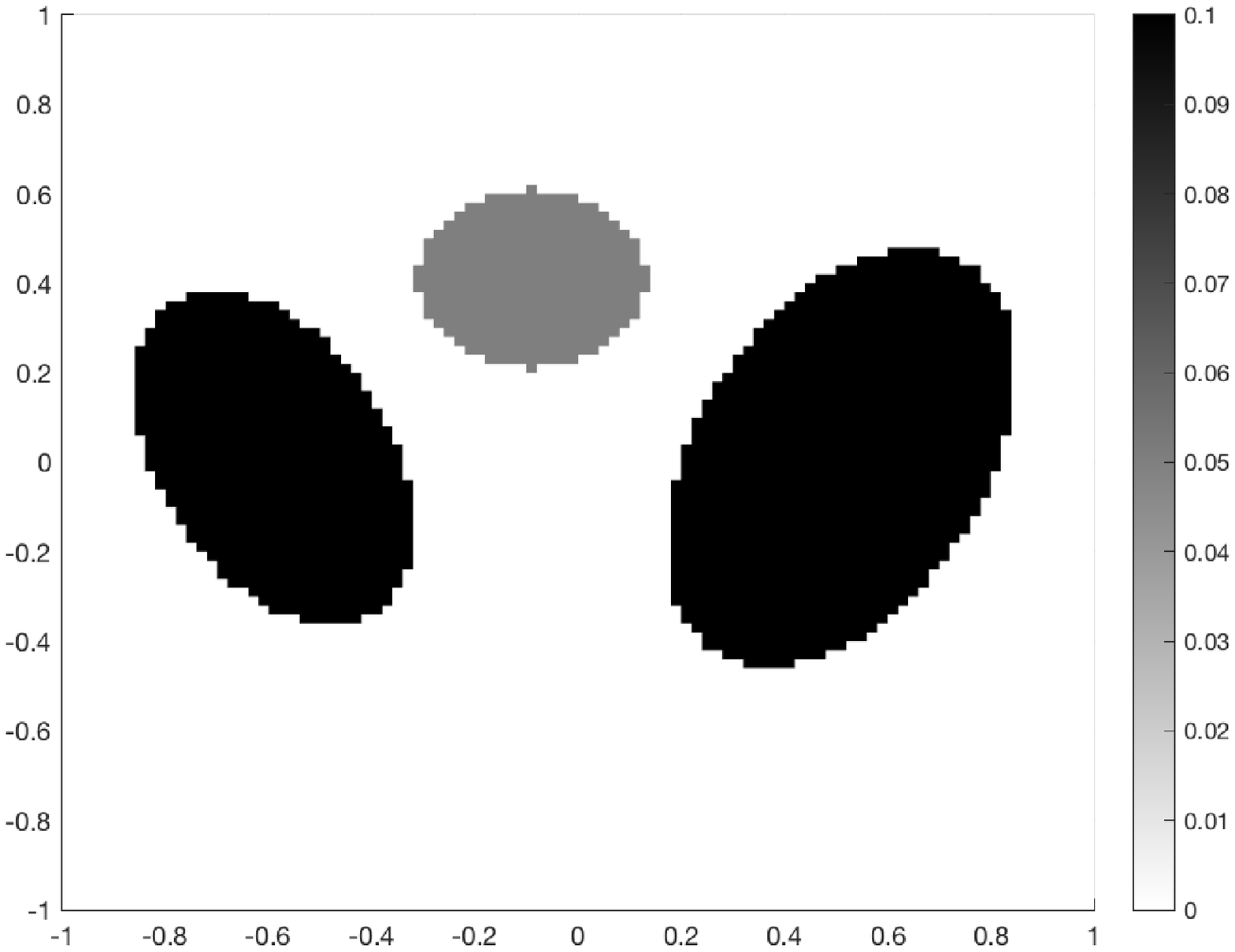}\label{mu_heart_exact}}
\subfloat[Reconstructed $\mu$]{\includegraphics[width=0.3\textwidth, height=0.25\textwidth]{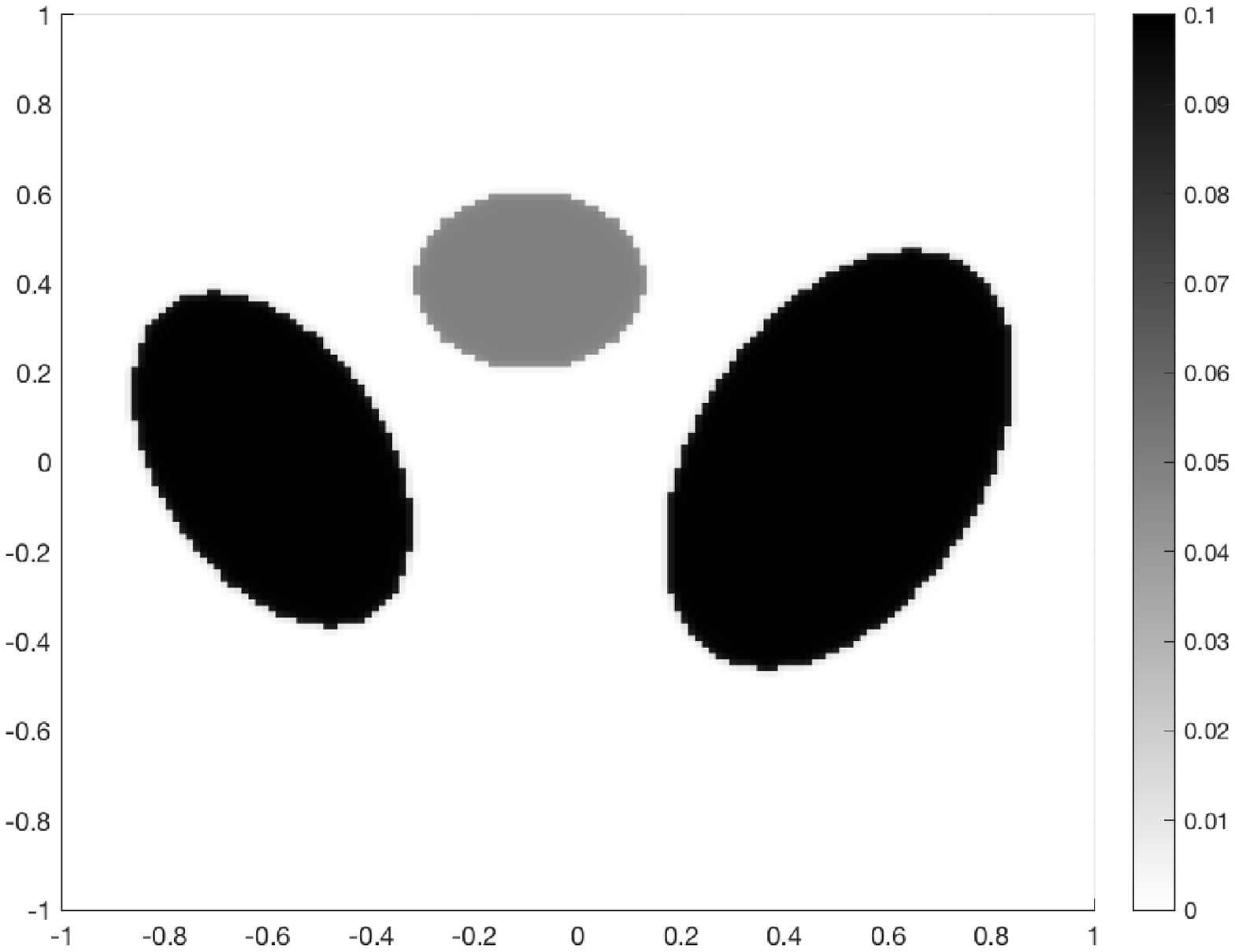}\label{mu_heart_recon}}
\subfloat[Reconstructed $\mu$ with 20\% noise]{\includegraphics[width=0.3\textwidth, height=0.25\textwidth]{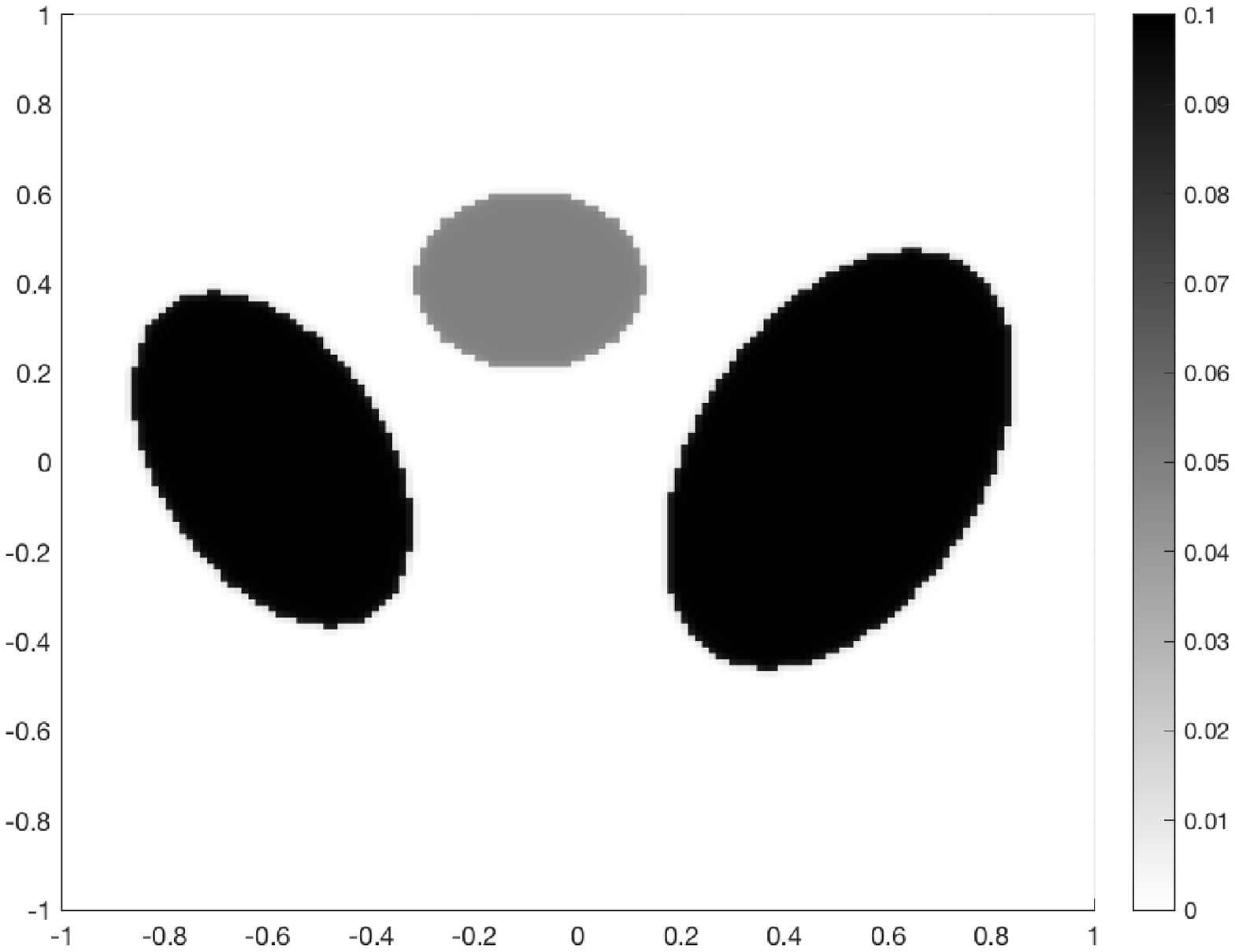}\label{mu_heart_recon_20}}\hspace{10mm}\\
 \caption{Test Case 2-Reconstructions of the heart and lung phantom with the 2PPAT-SR framework} 
    \label{heart}
  \end{figure}
  We again see from Figures \ref{sigma_heart_recon} and \ref{mu_heart_recon} that the reconstructions of $\sigma,\mu$ are of high contrast and high resolution. To test the robustness of our method, we add 20\% multiplicative Gaussian noise to the interior data $\mathcal{H}^{\sigma,\mu}$ and use it for our 2PPAT-SR inversion algorithm. We also modify the value of the regularization parameters $\xi_1=0.1,\xi_2=0.1,\gamma_1=0.3,\gamma_2=0.3$, in order to counter the noisy data. The results can be seen in Figure \ref{sigma_heart_recon_20} and \ref{mu_heart_recon_20}. We see that the reconstruction of $\sigma$ contains a few artifacts but still is of good quality. The reconstruction of $\mu$ demonstrates very little artifacts. This shows that our 2PPAT-SR reconstruction framework is robust and accurate even in the presence of noisy data.

In test case 3, we consider $\sigma$ as the Shepp-Logan phantom given in \cite{Shepp}. The background $\sigma_b$ is chosen to be 0.3 in this case. We compute $\mu=0.1\sigma$ and the background value of $\mu_b$ is chosen as 0.03. The plots of the exact and reconstructed phantoms are shown in Figure \ref{shepp}.

\begin{figure}[H]
\centering
\subfloat[Exact $\sigma$]{\includegraphics[width=0.3\textwidth, height=0.25\textwidth]{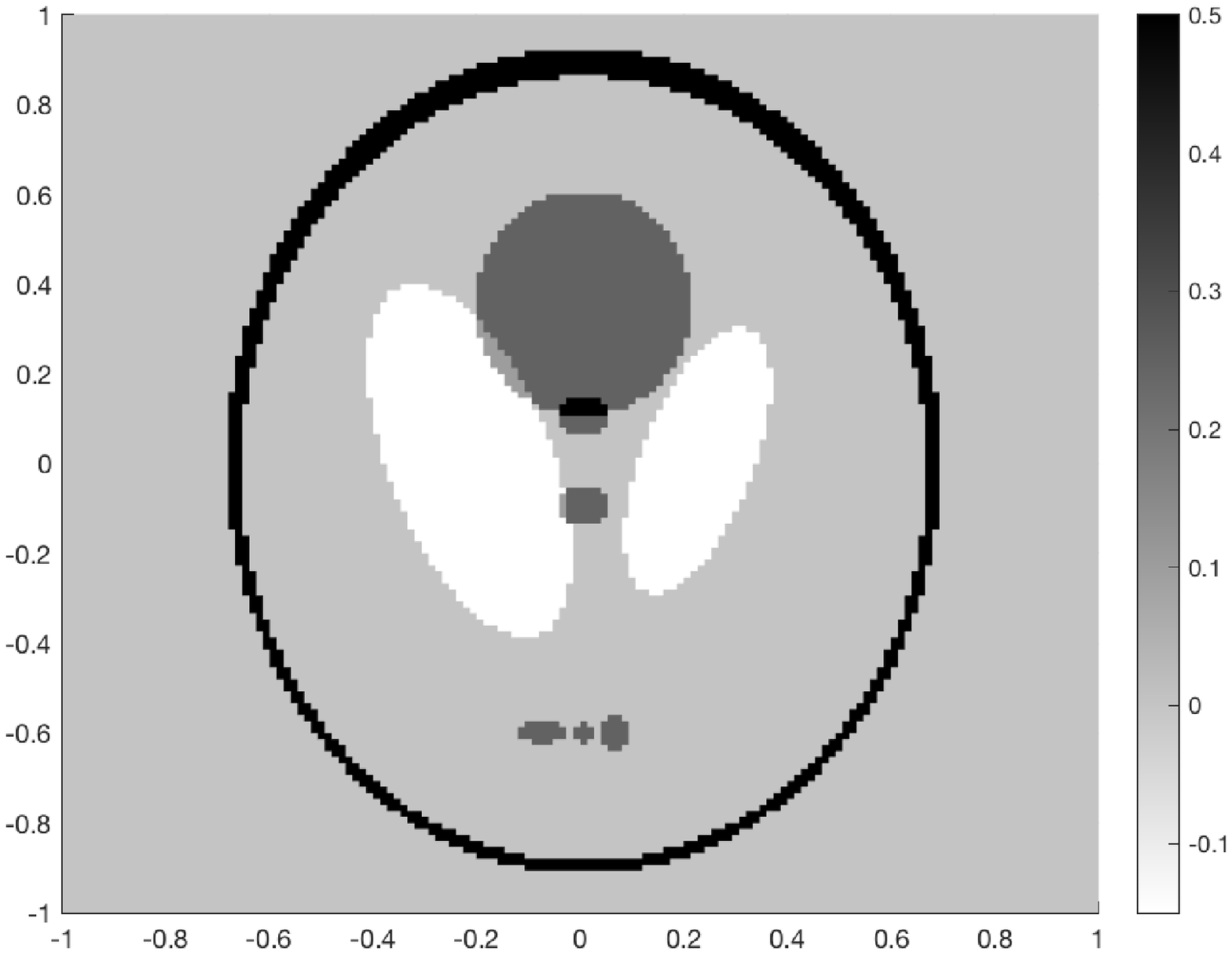}\label{sigma_shepp_exact}}
\subfloat[Reconstructed $\sigma$]{\includegraphics[width=0.3\textwidth, height=0.25\textwidth]{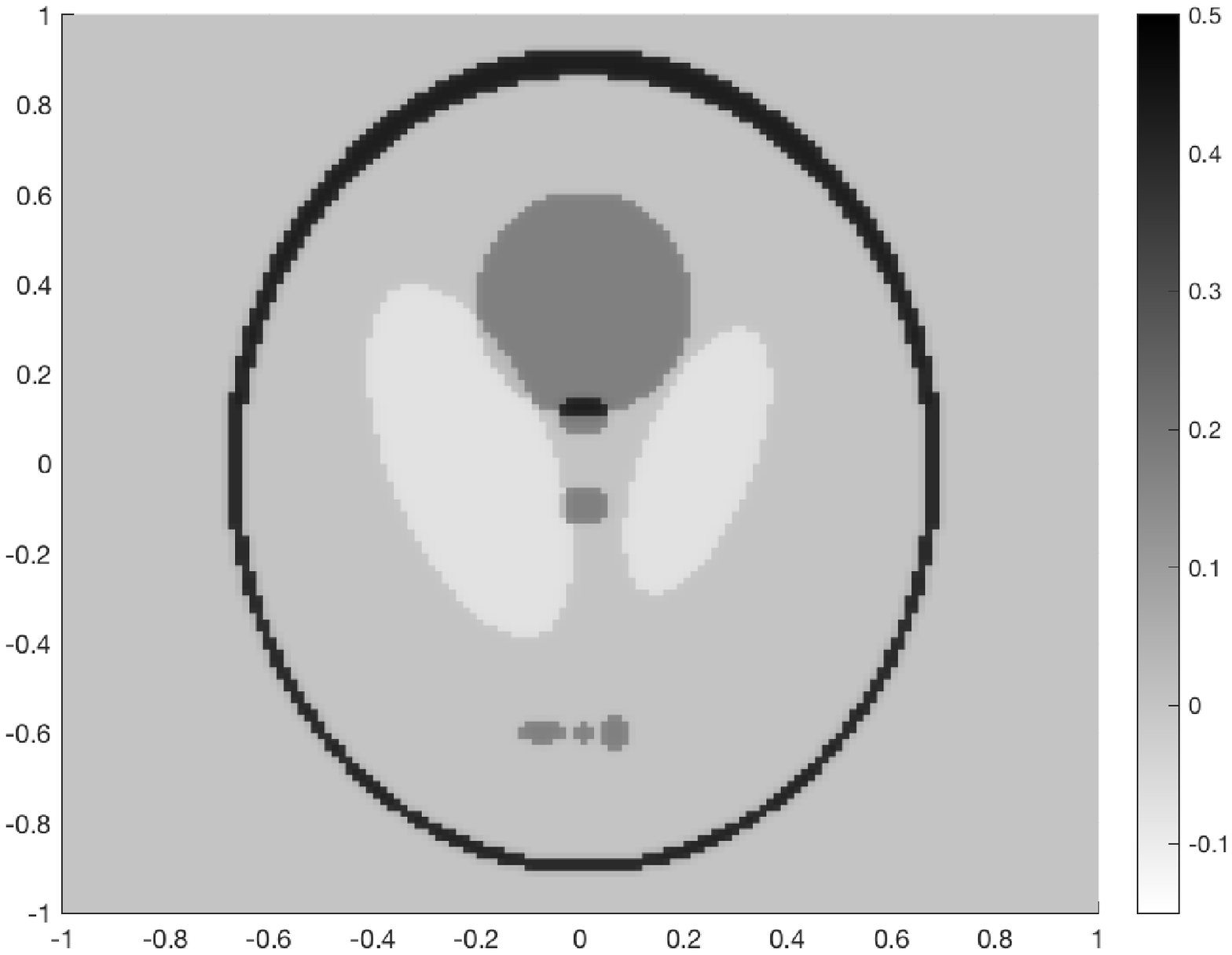}\label{sigma_shepp_recon}}
\subfloat[Reconstructed $\sigma$ with 20\% noise]{\includegraphics[width=0.3\textwidth, height=0.25\textwidth]{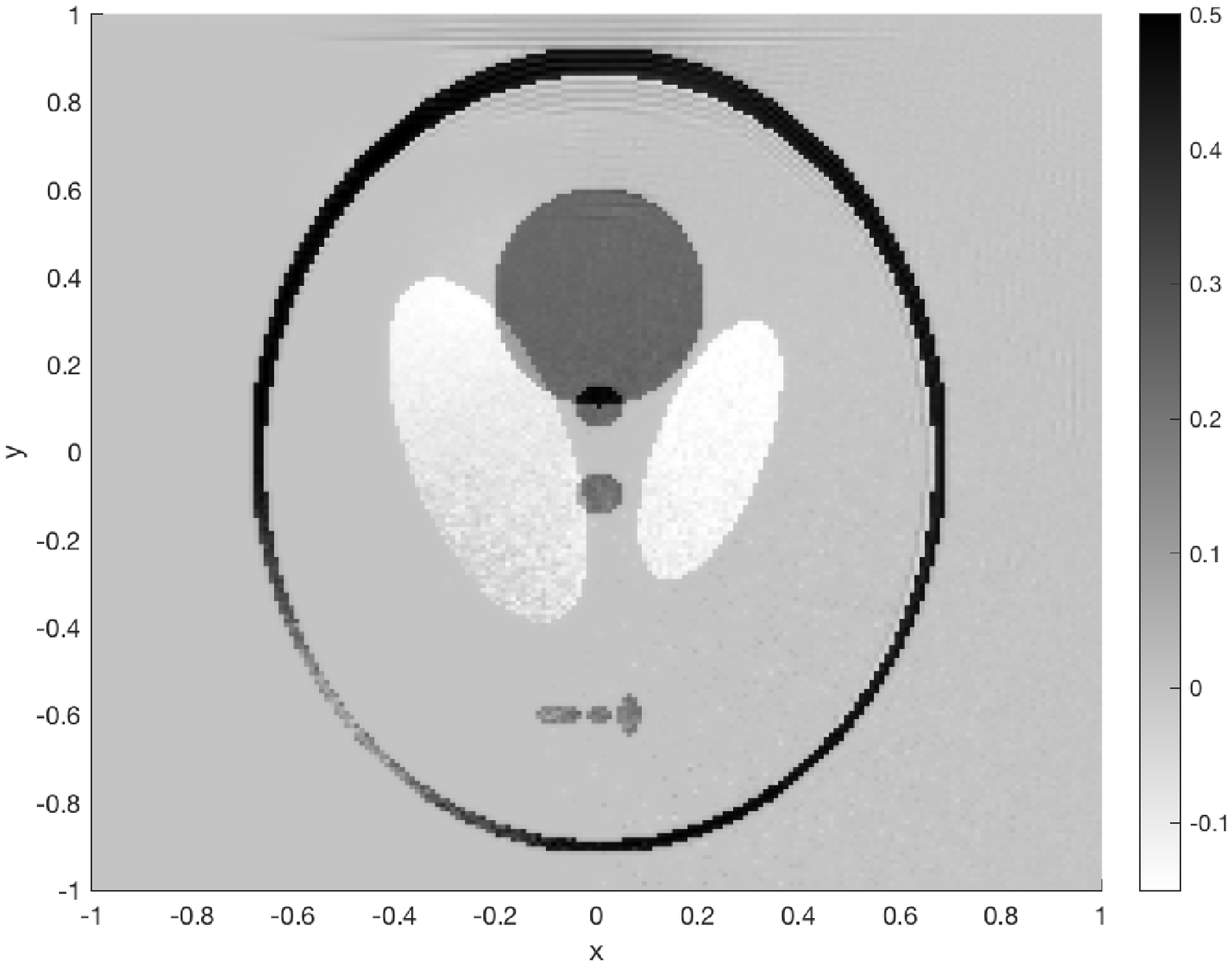}\label{sigma_shepp_recon_20}}\\
\subfloat[Exact $\mu$]{\includegraphics[width=0.3\textwidth, height=0.25\textwidth]{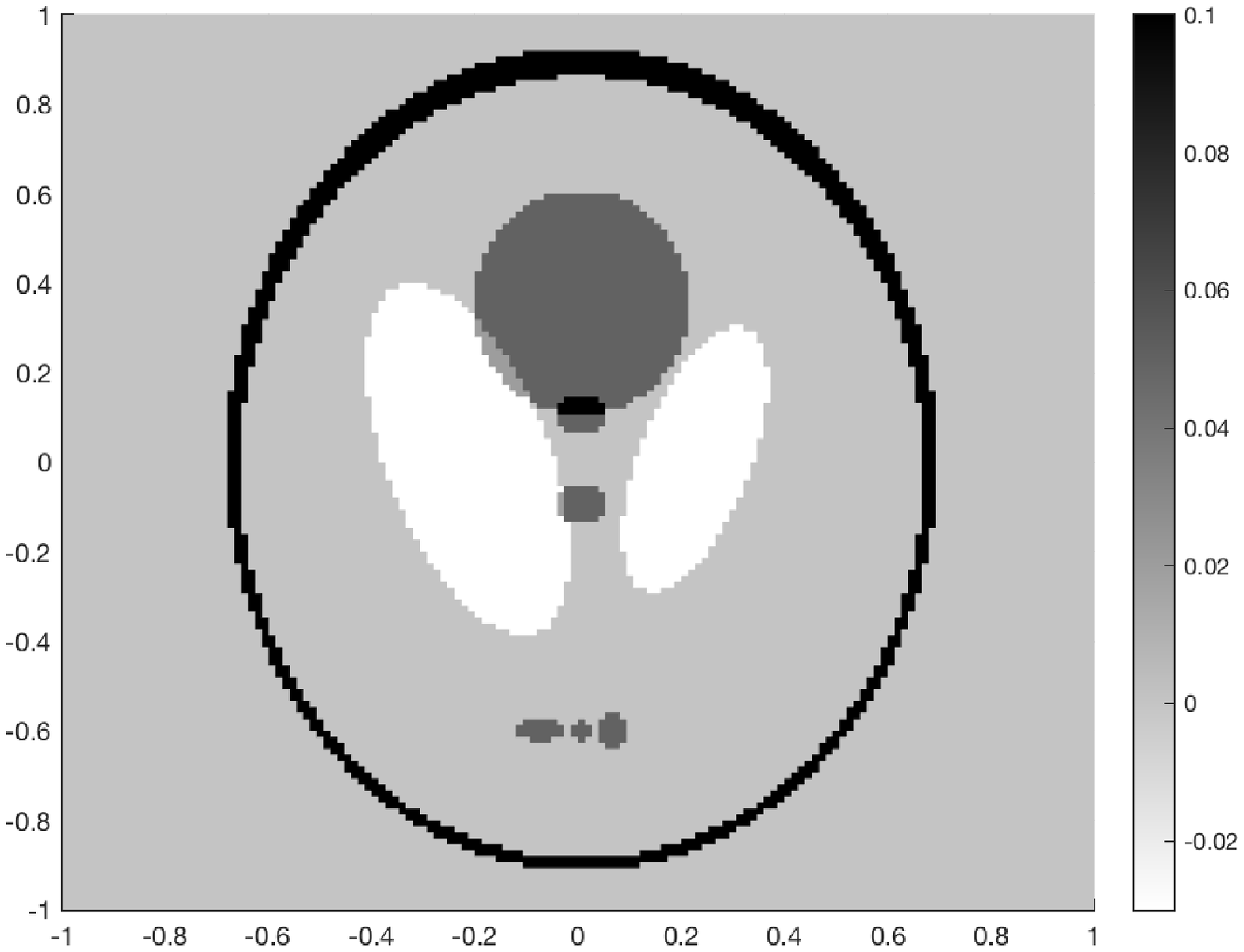}\label{mu_shepp_exact}}
\subfloat[Reconstructed $\mu$]{\includegraphics[width=0.3\textwidth, height=0.25\textwidth]{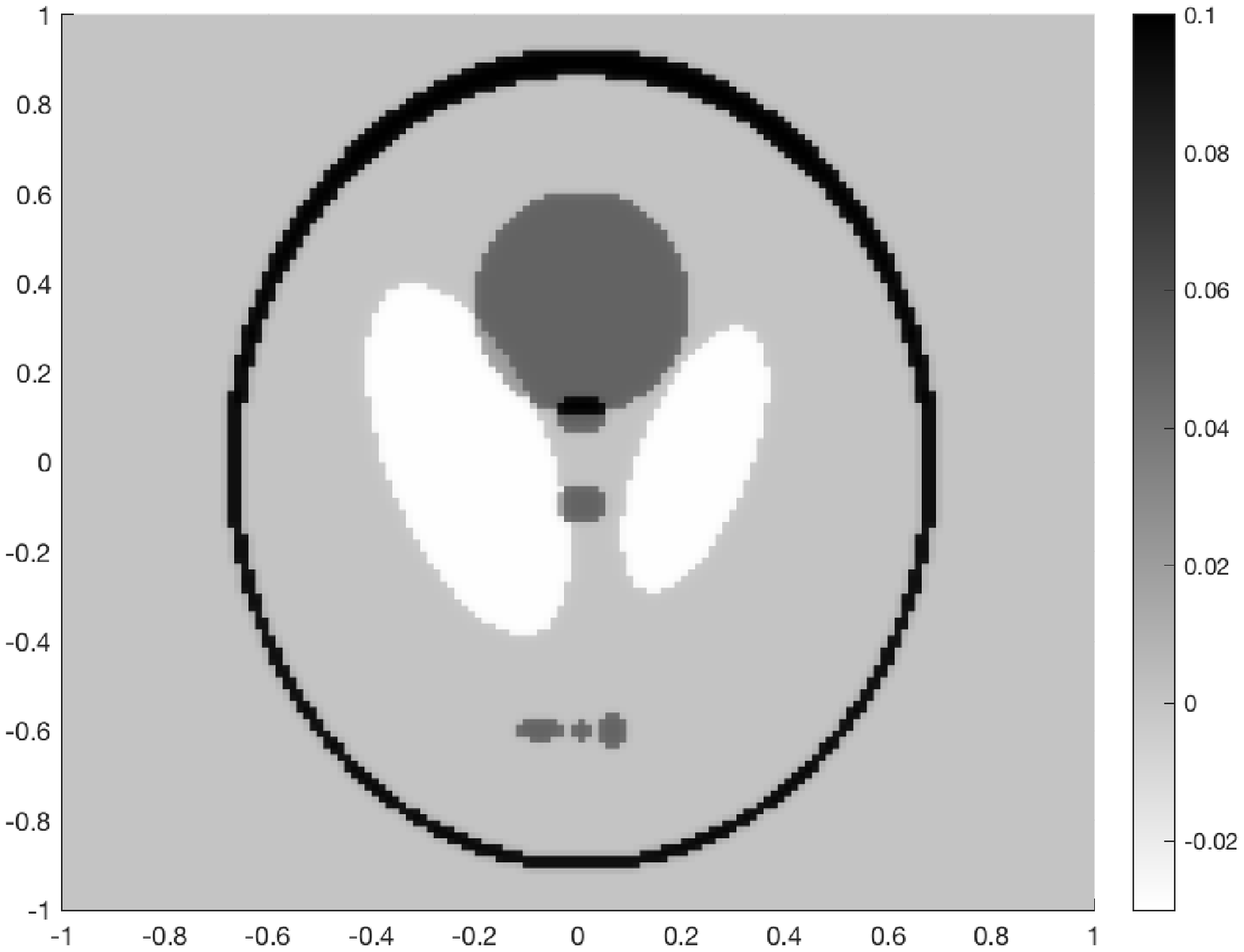}\label{mu_shepp_recon}}
\subfloat[Reconstructed $\mu$ with 20\% noise]{\includegraphics[width=0.3\textwidth, height=0.25\textwidth]{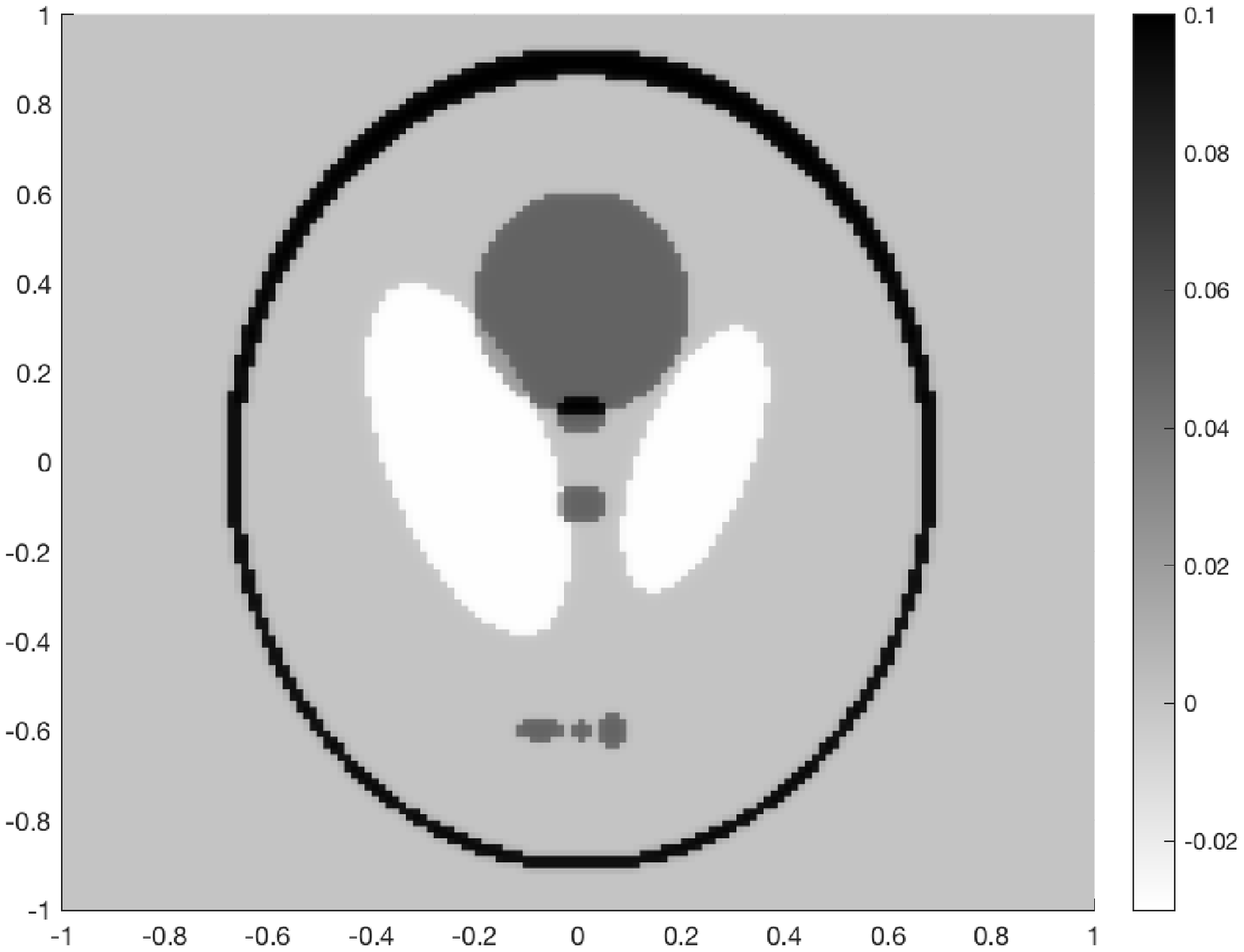}\label{mu_shepp_recon_20}}\hspace{10mm}\\
 \caption{Test Case 3-Reconstructions of the Shepp-Logan phantom with the 2PPAT-SR framework}
    \label{shepp}
  \end{figure}
  We again see from Figures \ref{sigma_shepp_recon} and \ref{mu_shepp_recon} that the 2PPAT-SR reconstruction framework gives superior quality reconstructions even for objects with high contrast values and with holes and inclusions. The reconstructions with 20\% noise in the interior data are shown in Figures \ref{sigma_shepp_recon_20} and \ref{mu_shepp_recon_20} with the modified regularization parameter values as in the previous test case. We see that the reconstructions are still of high quality with very less artifacts.

\section{Conclusions}
In this work, we have presented a new reconstruction framework for determining the optical coefficients in two-photon PAT. The framework comprises of a PDE-constrained optimization problem that promotes sparsity patterns in the reconstructions of the single and two photon absorption coefficients. We present a new theoretical analysis of the existence and uniqueness of a solution to a semi-linear elliptic PDE arising in 2P-PAT. Further, we present a proximal scheme using a Picard solver for the semi-linear PDE and its adjoint to solve the optimization problem. Several numerical results demonstrate that the proposed framework is able to achieve reconstructions with high contrast and high resolution for objects including holes and inclusions.

\section{Acknowledgments}
The authors are grateful to Gaik Ambartsoumian for several fruitful and important suggestions. S. Roy was partly supported by the National Cancer Institute, National Institutes of Health, grant number: 1R21CA242933-01.

%
\end{document}